\documentclass[aop]{imsart}

\RequirePackage{amsthm,amsmath,amsfonts,amssymb}
\RequirePackage[numbers]{natbib}

\usepackage{hyperref}

\startlocaldefs
\numberwithin{equation}{section}
\theoremstyle{plain}

\newtheorem{theorem}{Theorem}[section]

\newtheorem{lemma}[theorem]{Lemma}
\newtheorem{proposition}[theorem]{Proposition}

\theoremstyle{remark}

\newtheorem{remark}[theorem]{Remark}

\newcommand{\eps}{\varepsilon}
\newcommand{\vphi}{\varphi}

\newcommand{\calL}{\mathcal{L}}
\newcommand{\calF}{\mathcal{F}}

\newcommand{\calG}{\mathcal{G}}

\newcommand{\calM}{\mathcal{M}}
\newcommand{\calR}{\mathcal{R}}

\newcommand{\calS}{\mathcal{S}}

\newcommand{\calC}{\mathcal{C}}
\newcommand{\calB}{\mathcal{B}}

\newcommand{\calH}{\mathcal{H}}

\newcommand{\E}{\operatorname{\mathbb{E}}} 
\renewcommand{\P}{\operatorname{\mathbb{P}}} 
\newcommand{\R}{\mathbb{R}}
\newcommand{\mQ}{\mathbb{Q}}

\newcommand{\N}{{\mathbb{N}}}

\newcommand{\prt}{\partial}

\newcommand{\ol}{\overline}
\newcommand{\wh}{\widehat}
\newcommand{\wt}{\widetilde}

\DeclareMathOperator{\dist}{dist}

\newcommand{\bone}{\mathbf{1}}
\def\n{{\bf n}}

\newcommand{\set}[1]{\left\{#1\right\}}

\newcommand{\FF}{\mathbb F}
\newcommand{\QQ}{\mathbb Q}

\renewcommand{\H}{\mathbf H}
\newcommand{\cH}{\mathcal H}

\newcommand{\X}{\mathbf X}

\newcommand{\x}{\mathbf x}
\newcommand{\bfa}{\mathbf a}

\newcommand{\bQ}{\mathbf Q}

\newcommand{\cA}{\mathcal A}
\newcommand{\cL}{\mathcal L}

\endlocaldefs

\begin{document}

\begin{frontmatter}
\title{The Spine of the Fleming-Viot process\\ driven by Brownian motion}
\runtitle{Spine of Fleming-Viot process}

\begin{aug}
\author[A]{\fnms{Krzysztof} \snm{Burdzy}\ead[label=e1]{burdzy@uw.edu}}
\and
\author[B]{\fnms{J\'anos} \snm{Engl\"ander}\ead[label=e2,mark]{janos.englander@colorado.edu}}
\address[A]{Department of Mathematics, University of Washington, \printead{e1}}

\address[B]{Department of Mathematics, University of Colorado, \printead{e2}}
\end{aug}

\begin{abstract}
We show that the spine of the Fleming-Viot process driven by Brownian motion in a bounded Lipschitz domain with Lipschitz constant less than 1 converges to Brownian motion conditioned to stay in the domain forever.
\end{abstract}



\end{frontmatter}


\section{Introduction}\label{intro}

Under  suitable assumptions, a branching process can be decomposed into a spine and side branches. 
Heuristically speaking, the spine has the distribution of the driving process conditioned on non-extinction.
 We will prove this claim for the ``Fleming-Viot branching process'' introduced in \cite{BHM00}. In this paper, individuals follow independent Brownian motions and are killed on the boundary of a bounded  Lipschitz set.

A Fleming-Viot process is an extreme case of the Moran model introduced in \cite{Moran} (see \cite[Def.~5.12]{AE} for the modern discussion).
In the Moran model, individuals branch at a (bounded) intensity. In our model, 
an individual branches only when some other individual hits the ``boundary'' of the state space.

Our main result on the asymptotic spine distribution is limited to Fleming-Viot processes driven by Brownian motion. We conjecture that an analogous result holds for every Fleming-Viot process  (perhaps under mild technical assumptions). An analogous theorem was proved for Fleming-Viot processes driven by continuous time Markov processes on finite spaces in \cite{BiBu}.
That article contained extra results on the branching structure, missing from the present paper, namely, it was proved that the rate of branching along the spine converges to twice the rate for a generic particle
and  the distribution of a side branch converges to the distribution of a branching process with the critical branching rate. Weak convergence of spines to the driving process conditioned to survive forever has been recently proved in \cite[Cor. 5.3]{ToughArxiv} in the case when the driving process is a normally reflected diffusion in a compact domain with soft killing (see \cite[Sec. 1.3]{ToughArxiv}).

The literature on branching processes is huge so we will mention only a few key publications. Most of them contain extensive reference lists. 
A precursor of our model can be found in a paper by Moran \cite{Moran}. The book by Jagers \cite{Jagers75} is a classical treatise on branching processes and their applications to biology. A modern review of continuous time and space branching can be found in a book by Etheridge \cite{AE}.
The ``Evans' immortal particle picture'' was introduced in \cite{Evans}.
The ``look-down'' process was defined by Donnelly and Kurtz in \cite{DK96}.
Modern approaches to the spine can be found in \cite{EK2004} and \cite{Henard}.
Some of the most profound analysis of the genealogical structure of the Moran and related models appeared in \cite{GLW05,DGP12,GPW13,SeidelThesis}.

Our paper is organized as follows.
Our proof is  complicated so we start with a non-technical review of the proof strategy in Section \ref{j30.1}.
 Section~\ref{review} contains 
basic definitions and a review of known results. Section \ref{sec:main} contains the statement of Theorem \ref{a1.7}, our main result, and its proof. The proof is based on many estimates that are relegated to Section \ref{sec:est}.
We present a generalization of our main result to non-Brownian processes in Section \ref{sec: general}.
 Section \ref{sec:super} contains a short and informal review of the results on spines of superprocesses.

\section{Heuristic outline of the argument}\label{j30.1}

Consider a fixed population of $n\ge 2$ individuals that move independently according to Brownian motion. They all start inside a domain $\Lambda$ and are killed when they exit it. When an individual is killed, another individual, chosen randomly (uniformly), branches into two individuals. Therefore, the population size is always equal to $n$.

Assume that $\Lambda$ is a Lipschitz domain with a Lipschitz constant smaller than 1. It has been proven that there exists a single trajectory inside the branching structure, referred to as a ``spine,'' which extends over the entire time interval $[0,\infty)$ and never hits the boundary of $\Lambda$ (see Sections \ref{ssec:DHPs} and \ref{de2.1}).

Our argument involves ``Brownian motion conditioned to stay in $\Lambda$ forever.'' To clarify this concept, one can appeal to the main theorem in \cite{Pin}: the distribution of Brownian motion conditioned to stay in $\Lambda$ for a long time is close to that of a process obtained from Brownian motion by a space-time Doob's $h$-transform.  We state and prove our version of Pinsky's result in Lemma \ref{a14.3} since we require a specific order of quantifiers.

The main result of our paper is stated in Theorem \ref{a1.7}: as the population size $n$ increases to infinity, the spine processes converge in distribution to Brownian motion conditioned to stay in $\Lambda$ forever.

Below, we give a heuristic overview of the main ideas in the proof. 
The proof consists of multiple steps. Two of them are the following. Fix $\varepsilon>0$.

(A) (convergence of conditioned Brownian motions). For sufficiently large $s>0$, there is a $t_0=t_0(s)>s$ such that for all $t>t_0$, the following holds. 
Let BB (``Brownian bridge'') be  Brownian motion conditioned to stay in $\Lambda$ up to time $t$ and conditioned further to have a distribution $\nu$ at time $t$. Then, restricted to $[0,s]$, the (Prohorov) distance between the distribution of BB and that of Brownian motion conditioned to stay in $\Lambda$ forever is less than $\varepsilon$, uniformly over all probability distributions $\nu$ on $\Lambda$.

The second result we prove roughly states that for large $n$, the spine trajectory up to time $t$ is close to a particular Brownian bridge.

(B) (convergence of spines). Let $\nu_t^n$ denote the distribution of the $n$-spine (spine of the process with $n$ individuals) at time $t$ on $\Lambda$. Given $t>0$, for all sufficiently large $n$ (i.e., all $n>n_0(\varepsilon,t))$, the distribution of the trajectory of the $n$-spine on $[0,t]$ is closer than $\varepsilon$ to the distribution of Brownian motion conditioned to stay inside $\Lambda$ up to time $t$ and further conditioned on having distribution $\nu_t^n$ at time $t$.

Although we do not know $\nu_t^n$, Theorem \ref{a1.7} follows from (A) and (B) because the estimate in (A) is uniform in $\nu$.

Our proof of (B) relies heavily on \cite{Villemonais13}. The result in \cite{Villemonais13} states that when the population size $n$ is large, the distribution of the location of a randomly (uniformly) chosen individual at time $t$ is close to the distribution of the position at time $t$ of the driving process conditioned to stay in $\Lambda$ until time $t$. Villemonais' theorem applies to branching populations driven by very general Markov processes. This has been used in \cite{BiBu} to show that Villemonais' theorem can be applied not only to the positions of individuals at time $t$, but also to the whole trajectories (genealogies) of individuals alive at time $t$. In other words, if time $t>0$ is fixed and the population size $n$ is large, and one chooses an individual alive at time $t$ randomly (uniformly), then the distribution of the genealogy of this individual on $[0,t]$ is close to the distribution of the driving process on $[0,t]$ conditioned to stay in $\Lambda$ until time $t$.

Unfortunately, this very general theorem of Villemonais is not directly applicable in our setting. This is because, given the location of all individuals at time $t$, the individual residing on the spine is determined by the Fleming-Viot process on the interval $[t,\infty)$, and there is no reason why the spine position should be chosen uniformly among all individuals at time $t$ (in hindsight, our main theorem implies that it is not). Hence, in order to surmount this difficulty and make Villemonais' result applicable to our problem, we will need to consider a randomly picked particle with the (unknown to us) distribution of the spine at time $t$.

Villemonais proved a quantitative version of his theorem in \cite{Villemonais13}, which we  use to show that the genealogical trajectory of an individual chosen randomly (uniformly) in a small cube $Q\subset \Lambda$ at time $t$  has a distribution close to the distribution of Brownian motion conditioned to stay in $\Lambda$ on the time interval $[0,t]$ and reach $Q$ at time $t$.
Then we prove that, for a fixed small cube, the probability of the spine passing through any individual inside the cube at time $t$ is very close to the probability of it passing through any other individual present in the same cube at that time. This is due to the fact that, given the positions of individuals in a cube at time $t$ and their positions at a slightly later time $t+\Delta t$, Brownian motions that are conditioned to connect these sets of points are almost equally likely to choose any permutation---this follows from the form of the multidimensional Gaussian distribution. Therefore, on the scale of a single small cube, the spine position is chosen almost uniformly from all individuals present in the cube at time $t$ and, therefore, the Villemonais' estimate mentioned at the beginning of the paragraph can be applied.

We will now discuss a delicate aspect of the above reasoning. We condition the spine to be in $Q$ at a time $t+\Delta t>0$. The event we condition on, referring to the spine, is determined  by the evolution of the process only on $[t+\Delta t,\infty)$, and not on $[0,t)$. Therefore the conditioning has no bearing on the past behavior of the trajectories of the particles in $Q$. This is why for large $n$, the spine, conditioned this way is still close  (in distribution) to Brownian motion conditioned on not hitting the boundary before $t$ and ending up in $Q$ at time $t$.

To use the above to determine spine's position, we divide the domain $\Lambda$ into small cubes. Since we do not know where the spine passes at time $t$, all we can say is that the distribution of the spine on $[0,t]$ is close to the distribution of Brownian motion conditioned to stay in $\Lambda$ on $[0,t]$ and to have some (unknown to us) distribution at time $t$. But this is sufficient in view of (A).

In light of the preceding paragraphs, this completes the main argument for (B) and hence for Theorem \ref{a1.7}. However, there are several loose ends that need to be taken care of. 

First,   trajectories may branch on the small time interval $[t, t+\Delta t]$ mentioned above, but we show that this effect is negligible.

Second, the argument based on small cubes applies only to those cubes that are far from the boundary (relative to the cube size) so that we can assume that the trajectories of individuals starting from the cube at time $t$ have a negligible chance of hitting the boundary of $\Lambda$ during the time interval $[t,t+\Delta t]$. The cubes that are close to the boundary are addressed by proving that the spine is unlikely to be near the boundary of $\Lambda$. We accomplish this by showing that there are no trajectories in the branching structure that remain close to the boundary of $\Lambda$ for an extended period of time. Therefore, as the spine is one of these trajectories, it must stay away from the boundary for much of the time.

\section{Notation, definitions and known results}\label{review}

This section is based on \cite{BiBu}.

Our main theorem will be concerned with Fleming-Viot processes driven by Brownian motion in $\R^d$. Nevertheless we need to consider Fleming-Viot processes with an abstract underlying state space because our arguments will be based on ``dynamical historical processes'' which are Fleming-Viot processes driven by Markov processes with values in function spaces.  

Let $\Gamma$ be a topological space and let $\Lambda$ be a Borel proper subset of $\Gamma$. We will write
$\Lambda^c= \Gamma\setminus \Lambda$. 
Let $\{B_t,t\geq 0\}$ be a continuous time strong Markov process with state space $\Gamma$
whose almost all sample paths are right continuous.  For $s\geq 0$, let
\begin{equation}\label{m29.1}
  \tau_{\Lambda,s}=\inf\set{t>s:B_t\in \Lambda^c}, \qquad \tau_\Lambda = \tau_{\Lambda, 0}.
\end{equation}
We assume that $\Lambda^c$ is absorbing,
i.e., $B_t = B_{\tau_{\Lambda,s}}$ for all $t\geq \tau_{\Lambda,s}$, a.s.

We make the following assumptions.

(A1)	$ \P\left( s <\tau_{\Lambda,s}<\infty \mid B_s = x\right)=1 $ for all $x\in \Lambda$ and $s\geq 0$.

(A2) For every $x\in \Lambda$ and $s\geq 0$, the conditional distribution of $\tau_{\Lambda,s}$ given $\{B_s = x\}$ has no atoms. 

Consider an integer $n\geq 2$ and a family $\set{U_k^i,\, 1\leq i\leq n,\, k\geq 1}$ of jointly independent
random variables such that $U_k^i$ has the uniform distribution on the set
$\set{1,\dotsc,n}\setminus\set{i}$. 

We will use induction to construct a Fleming-Viot type process $\X^n_t=(X^1_t,\dotsc,X^n_t)$, $t\geq 0$, with
values in $\Lambda^n$. 
Let $\tau_0=0$ and consider  the (possibly random) initial configuration $(X_0^{1,1},\dotsc,X_0^{1,n})\in \Lambda^n$. Let 
\begin{equation}\label{o26.1}
X_t^{1,1},\dotsc,X_t^{1,n},\quad t\geq 0,
\end{equation} 
be independent and have the transition probabilities of the process $B$. We assume that processes in \eqref{o26.1} are independent of the family
$\set{U^i_k,\, 1\leq i\leq n,\, k\geq 1}$. Let 
\begin{equation*}
\tau_1=\inf\set{t>0:\exists_{1\leq i\leq n}\, X_t^{1,i}\in \Lambda^c}.
\end{equation*}
By assumption (A2), no pair of processes can exit $\Lambda$ at the same time, so the index $i$ in
the above definition is unique, a.s.

For the induction step, assume that the families 
\begin{equation*}
X_t^{j,1},\dotsc,X_t^{j,n}, \quad t\geq 0,
\end{equation*}
and the stopping times $\tau_j$ have been defined for $j\leq k$. For each $j\leq k$, denote by $i_j$ the unique
index such that $X_{\tau_j}^{j,i_j}\in \Lambda^c$. Let
\begin{equation*}
X_{\tau_k}^{k+1,m}=X_{\tau_k}^{k,m}\quad \text{for $m\neq i_k$,}
\end{equation*}
and
\begin{equation*}
X_{\tau_k}^{k+1,i_k}=X_{\tau_k}^{k,U^{i_k}_k}.
\end{equation*}
In words: it is the particle indexed by $(k,i_k)$, i.e., $X_{\cdot}^{k,i_k}$, which hits the boundary and then jumps on another, randomly chosen particle inside.
Let the conditional joint distribution of
\begin{equation*}
X_t^{k+1,1},\dotsc,X_t^{k+1,n},\quad t\geq \tau_k,
\end{equation*}
given
$\set{X_t^{j,m},\, 0\leq t \leq \tau_j, 1\leq m\leq n}$, $j\leq k$, and $\set{U_k^i,\, 1\leq i\leq n,\, k\geq 1}$,
be that of $n$ independent processes with transition probabilities of $B$, starting from $X_{\tau_k}^{k+1,m}$, $1\leq m\leq n$.
Let 
\begin{equation*}
\tau_{k+1}=\inf\set{t>\tau_k:\exists_{1\leq i\leq n}\, X_t^{k+1,i}\in \Lambda^c}.
\end{equation*}
We define $\X^n_t:=(X^1_t,\dotsc,X^n_t)$ by
\begin{equation*}
X_t^m=X_t^{k,m},\qquad \text{  for  } \tau_{k-1}\leq t<\tau_k,\  k\geq 1,\ 
m=1,2,\dotsc,n.
\end{equation*}
Note that the process $\X^n$ is well defined only up to the time
\begin{equation*}
\tau_\infty :=\lim_{k\to\infty} \tau_k.
\end{equation*}

We will say that $X^k$ experiences branching on the interval $[s_1,s_2]$ if some other process $X^m$ jumps from $\prt \Lambda$ to $X^k_s$ at a time $s\in[s_1,s_2]$. Note that the particle that is jumping is not considered to experience branching at that time.

If $B$ is Brownian motion then
the processes $X^m$, $m=1,\dots,n$, are ``driven'' by independent copies $B^m$ of $B$ that can be distilled from $X^m$'s in the following way:
\begin{align}\label{a22.1}
B^m_t= X^m_t - X^m_0 - \sum_{\tau_i \leq t} \left(X^m_{\tau_i} - X^m_{\tau_i-}\right), \qquad t\geq 0.
\end{align}

\subsection{Dynamical historical processes and spine}\label{ssec:DHPs}

Heuristically speaking,
for each $k\in\set{1,\dotsc,n}$, 
the ``dynamical historical process''
$\{H^k_t(s), 0\leq s\leq t\}$ (to be defined rigorously below) represents
the unique path in the branching structure of the Fleming-Viot process which goes from $X^k_t$ to one of the points $X^1_0, \dots, X^n_0$ along the trajectories of $X^1, \dots, X^n$ and does not jump at times $\tau_k$.

Let $\cA$ be the family of  all sequences of the form $((a_1,b_1), (a_2,b_2), \dots, (a_k,b_k))$, where $a_i \in \{1,\dots, n\}$ and $b_i \in \N$ for all $i$.
If $\bfa=((a_1,b_1), (a_2,b_2), \dots, (a_k,b_k))$ then we will write $\bfa + (m,r)$ to denote
$((a_1,b_1), (a_2,b_2), \dots, (a_k,b_k), (m,r))$. 
We will define a function
$\cL :  \{1,\dots, n\} \times [0,\tau_\infty) \to \cA$.
We interpret $\cL(i,s)$ as a label of $X^i_s$ so, by abuse of notation, we will write $\cL(X^i_s)$ instead of $\cL(i,s)$.
We let $\cL(X^i_s)=((i,0))$ for all $0\leq s < \tau_1$ and $1\leq i\leq n$. If $\cL(X^i_s)=\bfa$ for $\tau_{k-1} \leq s < \tau_k$,
$i \ne i_k$ and $i \ne U^{i_k}_k $ then we let $\cL(X^i_s)=\bfa$ for $\tau_{k} \leq s < \tau_{k+1}$. Suppose that $i = U^{i_k}_k $ and  $\cL(X^i_s)=\bfa$ for $\tau_{k-1} \leq s < \tau_k$. Then we let
$\cL(X^i_s)=\bfa+(i, k)$ and $\cL(X^{i_k}_s)=\bfa+(i_k,k)$ for $\tau_{k} \leq s < \tau_{k+1}$.

Suppose that $\cL(X^\ell_t)=((a_1,b_1), (a_2,b_2), \dots, (a_k,b_k))$ 
for some 
$k\geq 1$. 
The assumption (A1) on the driving process $B$ implies that $X^\ell$ will ``hit'' $\Lambda^c$ at some time greater than $t$ (more precisely, $\ell =i_j $ for some $j> b_k$) with probability 1.
Before that time, it may also happen that some other $X^i$ will jump onto $X^\ell$; more precisely, it may happen that $\ell =U^{i_j}_j $ for some $j> b_k$. Let $\tau'$ be the minimum of all such times.
From the definition of $\cL$ we easily infer that 
$0=b_1<b_2<\dotsc<b_k$ and $\tau_{b_k}\leq t$, so that 
$0<\tau_{b_1}<\dotsc<\tau_{b_k}\leq t<\tau'$. For 
$\tau_{b_m}\leq s<\tau_{b_{m+1}}$ with $1\leq m<k$ we let 
 $H^\ell_t(s)=X^{a_m}_s$, and for $\tau_{b_k}\leq 
s\leq t$ we 
let $H^\ell_t(s)=X^{a_k}_s$.

We will call $\{H^\ell_t(s), 0\leq s\leq t\}$ a \emph{dynamical historical 
process} (DHP) corresponding to $X^\ell_t$. Note that $H^\ell_t$ is defined 
 for $1\leq \ell \leq n$ and $0\leq t < \tau_\infty$.

The spine process will be defined below the statement of Theorem \ref{thm:uniquespine}. Roughly speaking, the spine is the unique DHP that extends from time 0 to time $\tau_\infty$.
The existence and uniqueness of the spine
was proved in \cite[Thm.~4]{GK}  under very
restrictive assumptions on the driving process $B$ and under the assumption that
the lifetime $\tau_\infty$ is infinite. It was proved in \cite{BiBu} that
the claim holds under  minimal reasonable assumptions, that
is, the strong Markov property of the driving process and non-atomic
character of the exit time distributions. 

\begin{theorem}\label{thm:uniquespine}
	Fix some $n\geq 2$, suppose that $B$ satisfies assumptions (A1)-(A2) and $\X^n_0 \in \Lambda^n$, a.s. Then, a.s.,
	there exists a unique infinite sequence $((a_1,b_1), (a_2,b_2), \dots)$ such that its every finite initial subsequence is equal to $\cL(X^i_s)$ for 
	some $1\leq i \leq n$ and $s\geq 0$. 
	
\end{theorem}

In the notation of the theorem, we define the spine of $\X^n$ by 
$J^n(s) = X^{a_m}_s$ for $\tau_{b_m}\leq s<\tau_{b_{m+1}}$, 
$m\geq 1$. We will write $\chi(s) = a_m$.

\subsection{Brownian motion-driven Fleming-Viot process}\label{de2.1}

From now on we will assume that the driving process $B$ is Brownian motion in $\R^d$.
We will assume that $\Lambda\subset \R^d$ is an open bounded connected Lipschitz domain with 
the Lipschitz constant less than 1. This means that every point in $\prt \Lambda$ has a neighborhood where $\prt \Lambda$ can be represented as the graph of  a Lipschitz function with the Lipschitz constant less than 1 in some orthonormal coordinate system.
Under these assumptions  $\tau_\infty= 
\infty$, a.s. (see \cite{BBF, GK}). We mention parenthetically that if the driving process is Brownian 
motion, $\Lambda$ is a polytope and $n=2$  then we also have $\tau_\infty= 
\infty$, a.s. (see \cite{BBF}). However, it was proved in \cite{BBP} that $\tau_\infty 
< \infty$, a.s., for every $n$, for some Fleming-Viot processes 
driven by one-dimensional diffusions. 

\subsection{Dynamical historical process as  a Fleming-Viot process}\label{DHP}

Let $C([0,t],\Gamma)$ denote the space of continuous functions with values in $\Gamma$, with the supremum norm. For a function $f: C([0,t],\Gamma)\to \R$, let
\begin{align}\label{m29.4}
\cH^n_t(f)&=\frac{1}{n}\sum_{k=1}^n f(H^k_t).
\end{align}

Let $\mu_n=\frac 1 n\sum_{k=1}^n \delta_{X^k_0}$, i.e., $\mu_n$ denotes the empirical distribution of $\X_0^n$.

Recall  definition \eqref{m29.1} and let 
\begin{align}\label{m29.3}
\wt \P_t^\mu(A) = \P\left(\{B_s, 0\leq s \leq t\} \in A\mid \tau_\Lambda > t\right),\qquad A \subset C([0,t],\Lambda),
\end{align}
assuming that $\P(B_0\in A_1)=\mu(A_1)$ for $A_1\subset \Lambda$. 
Note that this does not imply that $\wt \P_t^\mu(B_0\in A_1)=\mu(A_1)$ for $A_1\subset \Lambda$.

In the case when $\mu=\mu_n$, we will write $\wt \P_t$ instead of $\wt \P_t^{\mu_n}$.
The corresponding expectations will be denoted $\wt \E_t$ and $\wt \E_t^\mu$. We will write $\wt\E_t(f )$ instead of $\wt\E_t(f(\{B_s, 0\leq s \leq t\}) )$.

The following theorem is a corollary of \cite[Thm.~2.2]{Villemonais13}. 

\begin{theorem}\label{a3.2}
	For $t\geq 0$ and any measurable function $f: C([0,t],\Gamma)\to \R$ with $\|f\|_\infty\leq 1$,
\begin{align*}
\E
\left|\cH_t^n(f) -\wt\E_t(f )\right|
\leq  2\left(1+\sqrt{2}\right)
 \left(\E\left( \P_{\mu_n} (\tau_\Lambda > t)^{-2}\right)\right)^{1/2}
n^{-1/2},
\end{align*}
where $\P_{\mu_n}$ represents the distribution of the driving process $B$ with the initial distribution $\mu_n$. 
\end{theorem}

\begin{proof}
It has been shown in the proof of \cite[Thm. 4.2]{BiBu} that DHP can be identified with a space-time time-homogeneous Fleming-Viot process. Hence, our theorem follows directly from 
 \cite[Thm. 2.2]{Villemonais13}.
\end{proof}

\subsection{Conditioned Brownian motion}
A major monograph discussing conditioned Brownian motion is \cite{Doob}. The topic and the book are rather technical so the reader may find the presentation of the basic facts about conditioned Brownian motion and conditioned space-time Brownian motion in the introduction of \cite{BS} more accessible.

Recall that $\Lambda\subset \R^d$ is a bounded Lipschitz domain.
Let $\vphi>0$ denote the first eigenfunction of $(-\frac 1 2) \Delta$, where  $\Delta$ denotes the Dirichlet Laplacian in $\Lambda$. Let $\lambda>0$ be the corresponding eigenvalue. Then the space-time Brownian motion $(B_t,t)$ conditioned by the parabolic function $h(x,t)= e^{\lambda t} \vphi(x)$ stays in $\Lambda \times \R$ forever. The spatial component of this process can be considered to be ``Brownian motion  conditioned to stay in $\Lambda$ forever'' because it is the  weak limit, as $t\to \infty$, of Brownian motions conditioned not to exit $\Lambda$ in $[0,t]$ (see \cite{Pin}).
We will use $\wh \P^\mu$ to denote the distribution of Brownian motion  conditioned to stay in $\Lambda$ forever, with the initial distribution 
$c \vphi(x)\mu(dx)$, where $c>0$ is the normalizing constant. 

\section{Weak convergence of spines}\label{sec:main}

Recall that we assume that the driving process $B$ is Brownian motion 
and $\Lambda\subset \R^d$ is an open bounded connected Lipschitz domain with 
the Lipschitz constant less than 1. 

Our main result is as follows.

\begin{theorem}\label{a1.7}
Suppose that $\mu$ is a probability measure supported in a set $\Lambda_1\subset \Lambda$ such that $\dist(\Lambda_1, \Lambda^c) >0$.
Consider a sequence of Fleming-Viot processes $\X^n$ in $\Lambda$ driven by Brownian motion.
Assume that the measures $\mu_n=\frac 1 n\sum_{k=1}^n \delta_{X^k_0}$ are supported in $\Lambda_1$ and converge weakly to $\mu$ as $n\to \infty$. Then the distributions of spines $\{J^n_t, t\geq 0\}$ converge  to $\wh \P^\mu$.
\end{theorem}

\begin{remark}
(i) We believe that the theorem holds for all bounded Lipschitz domains, not only those with the Lipschitz constant less than 1. At present it is not known whether the lifetime $\tau_\infty$ is finite for the Fleming-Viot process driven by Brownian motion in any Lipschitz domain. However, it is implicit in arguments in \cite{BHM00} that for every Euclidean domain $\Lambda$ (not necessarily Lipschitz), $\tau_\infty\to\infty$ in distribution as the number $n$ of particles goes to infinity. This is enough to extend Theorem \ref{a1.7} to all Lipschitz domains. We omit the proof of this stronger result to avoid another layer of technicalities. We cannot get rid of the assumption of the Lipschitz character of the domain $\Lambda$ because it is an essential ingredient in Lemma \ref{a11.4}.

(ii) The proof of Theorem \ref{a1.7} is based on a large number of estimates specific to Brownian motion. While the overall structure of the proof, outlined in Section \ref{j30.1}, is very general and, therefore, it could be applied to any driving Markov process, Brownian motion is the only process for which the estimates needed in the proof are readily available in the literature, to our best knowledge.

(iii) One can generate examples of driving processes  for which the theorem holds by ``relabeling'' the state space as follows. 
If $\Lambda'$ is a set and $\FF: \Lambda \to \Lambda'$ is a one-to-one function then $\{\FF(\X^n_t), t\geq 0\}$, $n\geq 1$, are Fleming-Viot processes with the state space $\Lambda'$, driven by the image of Brownian motion by $\FF$.
There are some technical details that need to be taken care of. We show how this can be done in some specific cases in Section \ref{sec: general}.

(iv)
We assumed that the measures $\frac 1 n\sum_{k=1}^n \delta_{X^k_0}$ are supported in $\Lambda_1\subset \Lambda$ such that $\dist(\Lambda_1, \Lambda^c) >0$ because, to apply Theorem \ref{a3.2} in our argument, we need the following bound: for each fixed $t>0$,
\begin{align}\label{a8.1}
\limsup_{n\to\infty}
\left(\E\left( \P_{\mu_n} (\tau_\Lambda > t)^{-2}\right)\right)^{1/2}
<\infty.
\end{align}
It is easy to see that for every fixed $t>0$, the function $x\to \P (\tau_\Lambda > t\mid B_0 = x)$ is continuous and strictly positive  inside $\Lambda$. Hence,  
$\inf_{x\in \Lambda_1} \P (\tau_\Lambda > t\mid B_0 = x) >0$. This implies \eqref{a8.1}.

The following  example shows that \eqref{a8.1} fails for some natural initial distributions.
Suppose that $\Lambda$ has a smooth boundary, $\mu$ is the uniform probability distribution in $\Lambda$
and $X^k_0$, $k=1,\dots,n$, are i.i.d. with the distribution $\mu$.
By an argument similar to that in Lemma \ref{a11.4},
\begin{align*}
\P(\tau_\Lambda > t \mid B_0 = x) \leq c_1 \dist(x,\prt \Lambda),
\qquad x\in \Lambda,
\end{align*}
where $c_1$ depends on  $\Lambda$ and $t$. 
It follows that
\begin{align*}
\E\left( \P_{\mu_n} (\tau_\Lambda > t)^{-2}\right)
&= \frac 1 {|\Lambda|} \int_\Lambda \P(\tau_\Lambda > t \mid B_0 = x)^{-2} dx
\geq \frac 1 {|\Lambda|} \int_\Lambda c_1^{-2} \dist(x,\prt \Lambda)^{-2} dx\\
&\geq c_2 \int_{0+}  s^{-2} ds=\infty. 
\end{align*}
Therefore, \eqref{a8.1} fails in this case.

(v) It is conceivable that for a fixed finite number of individuals, the spine of the Fleming-Viot process has the distribution of the process conditioned not to hit the boundary of the domain. Although this seems to be highly unlikely, proving that this is not the case does not seem to be easy. Apparently there are only two examples showing that this is not the case, one in \cite[Sect. 6]{BiBu} and another one in \cite{BuTa}. They both deal
with Fleming-Viot processes with only $n=2$ individuals. In the first case, the process has a finite state space; in the second case, the driving process is Brownian motion in $[0,\infty)$.

(vi) We will use an informal notation for  Radon-Nikodym derivatives. Here is a typical example: $\P(B'_t \in dx)/\P(B''_s\in dx)$. In this example $B'$ and $B''$ are processes with values in $\R^d$, and $x\in \R^d$. The ratio of two ``probabilities'' represents the value of the Radon-Nikodym derivative of the distribution of $B'_t$ on $\R^d$ with respect to the distribution of $B''_s$ on $\R^d$, evaluated at $x$. Most applications of this notation will be more complex than this simple example but the interpretation will remain the same.
\end{remark}

\begin{proof}[Proof of Theorem \ref{a1.7}]
The proof is based on a large number of estimates that are relegated to Section \ref{sec:est}.

It will suffice to show that for every fixed $t_1>0$,
the distributions of $\{J^n_t,  0\leq t\leq t_1\}$ converge weakly to 
$\wh \P^\mu$ truncated in the obvious way to the interval $[0,t_1]$. To see this, note that convergence of finite dimensional distributions on $[0,\infty)$ is implied by the same type of convergence on all compact subintervals of $[0,\infty)$. It follows from Proposition 1.5 in Chapter XIII in \cite{RevuzYor99} that tightness on $[0,\infty)$  is implied by  tightness on compact time intervals.  So we fix an arbitrary $t_1>0$.

Let $0 < \gamma<\infty$ be a constant satisfying the condition in Lemma \ref{a11.4}.
We will argue that there exist $\alpha, \delta,\xi >0$  satisfying the following conditions,
\begin{align}
&\alpha \leq (1/2 -2\delta+3\gamma \delta/4 )/(\gamma+2d), \label{a3.3}\\
&0<\xi< 2\alpha -3\delta/2  , \label{a7.1}\\
&\alpha>\delta.\label{a7.8}
\end{align}
Note that $d$ and $\gamma$ are fixed at this point.
Let $\delta_0>0$ be so small that $(1/2 -2\delta+3\gamma \delta/4 )/(\gamma+2d) \geq (1/4 )/(\gamma+2d)$ for $0<\delta \leq \delta_0$. Let $\alpha_0 = (1/4 )/(\gamma+2d)$. If $0<\delta\leq \delta_0$ and $0<\alpha\leq \alpha_0$ then \eqref{a3.3} holds. We let $\alpha= \min(\alpha_0,\delta_0)/2$ and $\delta = \alpha/2$ so that \eqref{a7.8} holds true. Finally we note that with this choice of $\alpha$ and $\delta$, $2\alpha -3\delta/2>0$ so we can find $\xi$ such that \eqref{a7.1} is satisfied.

Fix an arbitrarily small $\eps>0$.
Fix some $u=u(t_1,\eps)>t_1$ which is greater than $s_1$  in Lemma \ref{a11.4}, greater than $s_1$ in Lemma \ref{a11.1}, and greater than $s_1$ in Lemma  \ref{a14.3}.

We will now refer to the notation introduced before and in Lemma \ref{a3.6}. This includes $C_j$, an atypical event, and  $k_1$. Let  $n_1$ be as in Lemma \ref{a3.6}, relative to $\eps$ and $u$ fixed above. From now on we will consider only $n>n_1$. 
Let $t_2 = u + j n^{-2\alpha + \delta}$, where  $j\geq 0$ is the smallest integer such that $\P(C_j)\leq \eps$ in the notation of Lemma \ref{a3.6}. Note that $j$ and, therefore, $t_2$ depend on $n$. 
In the notation of Lemma \ref{a3.6},
$\Delta t = n^{-2\alpha + \delta}$, $k_1 = \lfloor 1/\Delta t\rfloor$ and $j\leq k_1$ so $t_2 \in[u,u+1]$.

Let  $t_3 = t_2 + n^{-2\alpha + \delta}$.

For $z= (z_1, \dots, z_d)\in \R^d$ let
\begin{align}\label{rev31.1}
Q(z,r) = \left\{(y_1,\dots,y_d)\in \R^d: \max_{1\leq k \leq d} |y_k - z_k|\leq r\right\}.
\end{align}
Let $\bQ$ be the family of all cubes $Q=Q((z_1, \dots, z_d), n^{-\alpha}/2)$ such that every $z_k$ is an integer multiple of $n^{-\alpha}$, and  
\begin{align}\label{m29.5}
\dist(Q, \Lambda^c) \geq  3 n^{-\alpha +3\delta /4}.
\end{align}

If $0< s < t$ and $\omega\in C[0,t]$ then $\omega|_{s}\in C[0,s]$ will denote the truncation of $\omega$ to the interval $[0,s]$.

Fix a continuous non-negative function $f:C[0,\infty) \to \R$ with $\|f\|_\infty \leq 1$, which depends only on the values of the process on $[0,t_1]$, i.e., if  $\omega', \omega'' \in C[0,\infty)$ and $\omega'|_{t_1} = \omega''|_{t_1}$  then $f(\omega') = f(\omega'')$. 
In the same spirit, we can apply $f$ to $\omega\in C[0,s]$ for any $s\geq t_1$. Assume that  $\wt \E_{t_1} (f) >0$. By Lemma \ref{rev20.1},  $\wt \E_{t_2} (f) >0$.

For $Q\in \bQ$, let 
$G_Q=\{\omega_{t_2} \in Q\}$  and $f_Q = f\bone_{G_Q}$.
Recall  definition \eqref{m29.4} and let
\begin{align}\label{m29.10}
A_1&=\bigcap_{Q\in \bQ}\left\{\left|\cH_{t_2}^n(\bone_{G_Q}) -
	\wt\P_{t_2}(G_Q) \right|
\leq n^{-\alpha d +\gamma(-\alpha+3\delta/4) -\delta}\right\},\\
A_2&=\bigcap_{Q\in \bQ}\left\{\left|\cH_{t_2}^n(f_Q) -
	\wt\E_{t_2}(f_Q )\right|
\leq n^{-\alpha d +\gamma(-\alpha+3\delta/4) -\delta}\right\}.\label{m31.1}
\end{align}

For a cube $Q\in\bQ$, let $\calM_Q=\{j: X^{j}_{t_2} \in Q \}$.

Let $|\calM_Q|$ denote the cardinality of $\calM_Q$.
We will use the notation $\calM_Q= \{i_1, \dots, i_N\}$.
Note that $N=N(Q)=|\calM_Q|$ is a random integer.  
By abuse of notation, $|\,\cdot\,|$ will  denote the Euclidean norm in \eqref{a2.3} below and later in the paper.
Recall \eqref{a22.1} and let  
\begin{align}
\label{a2.3}
A_3 &= \bigcap_{1\leq i \leq n} 
\left\{
\sup_{s,t\in[t_2,t_3]}\left |B^{i}_s-B^{i}_t\right| < 2 n^{-\alpha +3\delta /4}\right\},\\
A_4 &=A_1\cap A_2\cap A_3.\label{sd.1}
\end{align}

Let $\calF_t = \sigma\{\X^n_s,0\leq s\leq t\}$.
Let $\pi$ be a random (i.e., uniform) permutation of $\{1,\dots,n\}$, independent of $\X^n$ and
\begin{align*}
\calF_t^+ = \sigma\left\{\left(X^{\pi(1)}_s, X^{\pi(2)}_s,\dots,X^{\pi(n)}_s\right), s\geq t\right\}.
\end{align*}

Let  $\calG_1$ be the smallest $\sigma$-field generated by $\calF_{t_2}$, $\calF_{t_3}^+$, and  random sets $\bigcup_{j\in\calM_Q}\{X^{j}_{t_3}\}$ for all $Q\in\bQ$. The information in the $\sigma$-field $\calG_1$ includes  the information on the locations of all particles at all times $t\geq t_3$, but with labels of the particles missing (scrambled). For every $Q\in\bQ$, $\calG_1$ contains the information on the locations of $X^j_{t_3}$ for $j\in \calM_Q$, but once again with the information on the labels missing. As a consequence, $\calG_1$ contains the information about the location of the spine at time $t_3$, and, for every $Q\in\bQ$, on whether the spine is equal to $X^j_{t_3}$ for some $j\in \calM_Q$ but without specifying $j$.

In view of \eqref{m29.5}, if $j\in\calM_Q$, $Q\in \bQ$ and $A_3$ holds then $X^{j}$ does not hit $\prt \Lambda$ in the time interval $[t_2, t_3]$ and, therefore, it does not jump during this interval. 
Thus, 
the joint conditional distribution of  $\left\{ X^{i}_t - X^i_{t_2}, t_2\leq t \leq t_3\right\}$, $i\in\calM_Q$, given $A_4$ and $\calF_{t_2}$ is that of independent Brownian motions $\left\{ B^{i}_t-B^i_{t_2}, t_2\leq t \leq t_3\right\}$, $i\in\calM_Q$, conditioned by 
\begin{align*}
\bigcap_{i\in\calM_Q}
\left\{
\sup_{s,t\in[t_2,t_3]}\left |B^{i}_s-B^{i}_t\right| <2 n^{-\alpha +3\delta /4}\right\}.
\end{align*} 
Let
\begin{align*}
\Lambda'=\left\{(x^1,x^2,z^1,z^2)\in \Lambda^4: x^1,x^2\in Q, z^1,z^2\in\Lambda,
|x^m-z^\ell|  \leq 
n^{-\alpha +3\delta /4} \text{ for }m,\ell=1,2\right\}.
\end{align*}

We will explain how Lemma \ref{a2.5} implies that
\begin{align}\label{m29.7}
&\frac{\P\left(
\{X^{j}_{t_3} \in dz^1,X^{k}_{t_3} \in dz^2 , X^{j}_{t_2} \in d x^1, X^{k}_{t_2} \in d x^2\} \cap {A_4} \mid \calF_{t_2} \right)}
{\P\left(
\{X^{k}_{t_3} \in dz^1,X^{j}_{t_3} \in dz^2 ,  X^{j}_{t_2} \in d x^1, X^{k}_{t_2} \in d x^2 \} \cap {A_4}\mid \calF_{t_2}\right) }
\leq \exp\left( 4\sqrt{d}  n^{ -\delta/4}\right),
\end{align}
for all $(x^1,x^2,z^1,z^2)\in\Lambda'$,   $Q\in \bQ$, $ j,k \in \calM_Q$, and sufficiently large $n$. Processes $X^j$ and $X^k$ play the same role as $B'$ and $B''$ in Lemma \ref{a2.5}. 
The events $F(\,\cdot\,)$ that appeared in Lemma \ref{a2.5} are replaced (implicitly) in \eqref{m29.7} by $A_3$ because $A_3\subset A_4$. 
Compared to Lemma \ref{a2.5}, there is extra conditioning in \eqref{m29.7}, specifically on  the event $A_1\cap A_2\subset A_4$. The extra conditions are independent of the trajectories of $X^j$ and $X^k$ on the interval $[t_2,t_3]$ given $\calF_{t_2}$.

Let 
\begin{align}\label{rev30.2}
\Lambda''=\Lambda''(Q)=\{(z^1, \dots,z^N) \in \Lambda^N: |x-z^j|  \leq 
n^{-\alpha +3\delta /4} \ \text{for}\ j=1,\dots,N,\ x\in Q\}.
\end{align}

Consider  $x^1, \dots, x^N \in Q$ and  $(z^1, \dots,z^N) \in \Lambda''$.
We will argue that for any $j_1,j_2\in \calM_Q$, and sufficiently large $n$,
\begin{align}\notag
&\P\Big(
\{X^{j}_{t_2} \in d x^j,  j \in\calM_Q \}\cap \{
X^{j_1}_{t_3} \in  dz^{j_2} \}\cap \{
X^{j_2}_{t_3} \in  dz^{j_1} \}\\
&\qquad\qquad \cap \{
X^{j}_{t_3} \in  dz^j,  j \in\calM_Q\setminus\{ j_1, j_2\}
\} \cap {A_4} \mid \calF_{t_2} \Big)\notag\\
&\times \left[\P\left(
\{X^{j}_{t_2} \in d x^j,  j \in\calM_Q \}\cap \{
X^{j}_{t_3} \in  dz^j,  j \in\calM_Q\} \cap {A_4}
 \mid\calF_{t_2}\right)\right]^{-1} \notag \\
&\qquad\leq  \exp\left( 4\sqrt{d}  n^{ -\delta/4}\right).
\label{m29.8}
\end{align}
The last formula is a more elaborate version of \eqref{m29.7}.
Indices $j_1$ and $j_2$ in \eqref{m29.8} play the same role as $j$ and $k$ in \eqref{m29.7}.
 Processes
$X^{j}$ with $  j \in\calM_Q\setminus\{ j_1, j_2\}$ are also included in 
\eqref{m29.8}. This does not affect the validity of the formula because the extra conditioning related to these processes is the same in the numerator and denominator (recall that  Brownian motions $B^i$, $1\leq i \leq n$, are jointly independent). 

Consider once again $x^1, \dots, x^N \in Q$ and  $(z^1, \dots,z^N) \in \Lambda''$. 
Suppose that
\begin{align*}
\calM_Q&= \{j_1, \dots, j_N\},\\
\{x^1,x^2,\dots,x^N\}&=\{X^{j_1}_{t_2}, X^{j_2}_{t_2},\dots,X^{j_N}_{t_2}\},\\
\qquad \{z^1,z^2,\dots,z^N\}&=\{X^{j_1}_{t_3}, X^{j_2}_{t_3},\dots,X^{j_N}_{t_3}\}.
\end{align*}
For $1\leq i\leq N$,
let $\Theta(i)=r$ if there is $j_k$ such that $X^{j_k}_{t_3} = z^i$ and  
$X^{j_k}_{t_2} = x^r$. Note that $\Theta$ is a permutation of $\{1,\dots,N\}$.

Consider $1\leq i, r_1,r_2 \leq N$. Let $\Pi_\ell$ be the set of permutations $\sigma$ of $\{1,\dots,N\}$ such that $\sigma(i) = r_\ell$, for $\ell=1,2$. If $\sigma\in\Pi_1$ and $\sigma(i_1)=r_2$ then we let $\calR(\sigma)$ be a permutation such that $\calR(\sigma)(i) = r_2$, $\calR(\sigma)(i_1)=r_1$, and $\calR(\sigma)(k)=\sigma(k)$ for $k\ne i,i_1$. It is easy to see that $\calR:\Pi_1\to\Pi_2$ is a bijection.
It follows from \eqref{m29.8} that for every $\sigma\in \Pi_1$,
\begin{align*}
&\frac
{\P\Big(\{\Theta=\sigma \} \cap {A_4} \mid \calG_1 \Big)}
{\P\Big(\{\Theta=\calR(\sigma) \} \cap {A_4} \mid \calG_1 \Big)}
\leq  \exp\left( 4\sqrt{d}  n^{ -\delta/4}\right).
\end{align*}
The conditioning $\sigma$-field has changed from $\calF_{t_2}$ to $\calG_1$,
relative to \eqref{m29.8}. The latter $\sigma$-field contains information about the set of locations of $X^i$'s at time $t_3$ for $i\in\calM_Q$, without information about which 
location corresponds to which process $X^i$. Hence, the last formula compares probabilities of different ``permutations'' or assignments of $X^i$'s to different locations. 

Summing over all $\sigma\in \Pi_1$ in the last formula yields for every $i=1,\dots,N$ and $r_1,r_2=1,\dots,N$,  
\begin{align}
&\frac
{\P\Big(\{\Theta(i)=r_1 \} \cap {A_4} \mid \calG_1 \Big)}
{\P\Big(\{\Theta(i)=r_2 \} \cap {A_4} \mid \calG_1 \Big)}
\leq  \exp\left( 4\sqrt{d}  n^{ -\delta/4}\right).
\label{rev22.3}
\end{align}

Recall \eqref{rev30.2} and let
\begin{align}\label{rev.m6.1}
& A_5^Q(z^1, \dots,z^N)
=\left\{ \bigcup_{j\in\calM_Q}\left\{X^{j}_{t_3}\right\}
=\left\{z^1, \dots,z^N\right\}
\right\},\\
&A_5^Q= \bigcup _{(z^1, \dots,z^N) \in \Lambda''(Q)}A_5^Q(z^1, \dots,z^N),
\label{rev30.7}\\
&A_6^Q =A_4\cap  A_5^Q,\label{rev30.8}
\end{align}
and note that the event $A_5^Q(z^1, \dots,z^N)$ concerns the equality of two unordered sets.
The event $A_6^Q$ is the intersection of $A_4$ and the event that  processes $X^j$ for $j\in \calM_Q$ stay within a  distance $n^{-\alpha +3\delta /4}$ from all points of $Q$ at time $t_3$.

If $A_5^Q(z^1, \dots,z^N)$ holds then for  $1\leq m\leq N$,  let $k(m)$ be such that $X^{k(m)}_{t_3} = z^m$.
It follows from \eqref{rev22.3} that  for any $1\leq m_1,m_2\leq N$ and $ j_1,j_2 \in \calM_Q$,  and sufficiently large $n$, 
\begin{align}\label{a19.4}
\frac
{\P\left(\left\{H^{k(m_1)}_{t_3}(t_2) = X^{j_1}_{t_2}
\right\}\cap A_4\cap A_5^Q(z^1, \dots,z^N) \mid \calG_1\right)}
{\P\left(\left\{H^{k(m_2)}_{t_3}(t_2) = X^{j_2}_{t_2}
\right\}\cap A_4\cap A_5^Q(z^1, \dots,z^N) \mid \calG_1\right)}
\leq  \exp\left( 4\sqrt{d}  n^{ -\delta/4}\right).
\end{align}
Hence 
\begin{align*} \notag 
&\P\left(\left\{H^{k(m_2)}_{t_3}(t_2) = X^{j_2}_{t_2}
\right\}\cap A_4\cap A_5^Q(z^1, \dots,z^N) \mid \calG_1\right)\\
&\qquad\geq  \exp\left( -4\sqrt{d}  n^{ -\delta/4}\right)
\P\left(\left\{H^{k(m_1)}_{t_3}(t_2) = X^{j_1}_{t_2}
\right\}\cap A_4\cap A_5^Q(z^1, \dots,z^N) \mid \calG_1\right), \notag \\
&\int_{\Lambda''} \P\left(\left\{H^{k(m_2)}_{t_3}(t_2) = X^{j_2}_{t_2}
\right\}\cap A_4\cap A_5^Q(z^1, \dots,z^N) \mid \calG_1\right)
dz^1\dots dz^N \notag \\
&\qquad\geq  \exp\left( -4\sqrt{d}  n^{ -\delta/4}\right) \notag \\
&\qquad \times\int_{\Lambda''} \P\left(\left\{H^{k(m_1)}_{t_3}(t_2) = X^{j_1}_{t_2}
\right\}\cap A_4\cap A_5^Q(z^1, \dots,z^N) \mid \calG_1\right)
dz^1\dots dz^N, \notag \\
&\P\left(\left\{H^{k(m_2)}_{t_3}(t_2) = X^{j_2}_{t_2} 
\right\}\cap A_6^Q \mid \calG_1\right)\\
&\qquad\geq  \exp\left( -4\sqrt{d}  n^{ -\delta/4}\right)
\P\left(\left\{H^{k(m_1)}_{t_3}(t_2) = X^{j_1}_{t_2}
\right\}\cap A_6^Q \mid \calG_1\right). \notag 
\end{align*}
The last inequality implies that
\begin{align*}
&|\calM_Q|\P\left(\left\{H^{k(m_2)}_{t_3}(t_2) = X^{j_2}_{t_2}
\right\}\cap A_6^Q \mid \calG_1\right)
=\sum_{1\leq m_1 \leq N}\P\left(\left\{H^{k(m_2)}_{t_3}(t_2) = X^{j_2}_{t_2}
\right\}\cap A_6^Q \mid \calG_1\right)\\
&\qquad\geq  \exp\left( -4\sqrt{d}  n^{ -\delta/4}\right)
\sum_{1\leq m_1 \leq N}
\P\left(\left\{H^{k(m_1)}_{t_3}(t_2) = X^{j_1}_{t_2}
\right\}\cap A_6^Q \mid \calG_1\right)\\
&\qquad=\exp\left( -4\sqrt{d}  n^{ -\delta/4}\right)
\P\left( A_6^Q \mid \calG_1\right)
=\exp\left( -4\sqrt{d}  n^{ -\delta/4}\right)
\bone_{A_6^Q}.
\end{align*}
We have shown that  for all $1\leq m\leq N$, $ j \in \calM_Q$ and sufficiently large $n$, 
\begin{align}
\P\left(\left\{H^{k(m)}_{t_3}(t_2) = X^{j}_{t_2}
\right\}\cap A_6^Q \mid \calG_1\right)
\geq \frac 1 {|\calM_Q|} \exp\left(- 4\sqrt{d}  n^{ -\delta/4}\right)
\bone_{A_6^Q}.\label{a24.2}
\end{align}

Suppose that $R$ is an $ \calF_{t_3}^+$-measurable random variable  
taking values in  $\{X^1_{t_3}, \dots, X^n_{t_3}\}$
and let
\begin{align}\label{rev.m6.2}
A_7^Q=\bigcup_{j\in \calM_Q}\{R=X^j_{t_3}\}.
\end{align}
By abuse of notation, if $R=X^j_{t_3}=z^m$, we will write $H^R_{t_3}$ or $H^{z^m}_{t_3}$ instead of $H^j_{t_3}$. Note that $R$ and $A_7^Q$ are $\calG_1$-measurable so \eqref{a24.2} implies that 
\begin{align*}
\P\left(\left\{H^{k(m)}_{t_3}(t_2) = X^{j}_{t_2}
\right\}\cap A_6^Q \cap A_7^Q \mid \calG_1\right)
\geq \frac 1 {|\calM_Q|} \exp\left(- 4\sqrt{d}  n^{ -\delta/4}\right)
\bone_{A_6^Q\cap A_7^Q}.
\end{align*}
Recall that $f:C([0,\infty) )\to [0,1]$ is a continuous function  and it only depends  on the values of the process on $[0,t_1]$. We will explain how the last formula  and the definition \eqref{m29.4} imply that
\begin{align}\label{a24.1}
 \E&\left(\cH_{t_2}^n(\bone_{G_Q}) f\left(H^{R}_{t_3}\right) \bone_{A_6^Q\cap A_7^Q}\mid \calG_1\right)\\
&= 
 \sum_{j\in \calM_Q} 
\E\left(\frac {|\calM_Q|} n f\left(H^{j}_{t_3}\right) \bone_{A_6^Q\cap A_7^Q}\mid \calG_1\right)
\P\left(\left\{H^{R}_{t_3}(t_2) = X^{j}_{t_2}\right\}\cap A_6^Q\cap A_7^Q
\mid \calG_1\right) \notag \\
&\geq \frac  {|\calM_Q|}n
\sum_{j\in \calM_Q} 
\E\left( f\left(H^{j}_{t_3}\right) \bone_{A_6^Q\cap A_7^Q}\mid \calG_1\right)
\frac 1 {|\calM_Q|}\exp\left(- 4\sqrt{d}  n^{ -\delta/4}\right)
\bone_{A_6^Q\cap A_7^Q} \notag \\
&= \frac 1n
\sum_{j\in \calM_Q} 
\E\left( f\left(H^{j}_{t_3}\right) \bone_{A_6^Q\cap A_7^Q}\mid \calG_1\right)
\exp\left(- 4\sqrt{d}  n^{ -\delta/4}\right) \notag\\
&= \frac 1n
\E\left( \sum_{j\in \calM_Q}  f\left(H^{j}_{t_3}\right) \bone_{A_6^Q\cap A_7^Q}\mid \calG_1\right)
\exp\left(- 4\sqrt{d}  n^{ -\delta/4}\right) . \notag 
\end{align}
There is no indicator $\bone_{A_6^Q\cap A_7^Q}$ at the end of the last line because it is implicitly included in the conditional expectation, since $A_6^Q\cap A_7^Q$ is $\calG_1$-measurable.
The first equality follows from the ``total probability formula,'' i.e., we sum over all events $\{H^{R}_{t_3}(t_2) = X^{j}_{t_2}\}\cap A_6^Q\cap A_7^Q$ because they form a partition of $A_6^Q\cap A_7^Q$. Since $f$ depends only on the part of the trajectory in $[0,t_1]$, the paths of $H^{R}_{t_3}$ and $H^{j}_{t_3}$ agree on $[0,t_1]$ if $\{H^{R}_{t_3}(t_2) = X^{j}_{t_2}\}\cap A_6^Q\cap A_7^Q$ holds. The quotient $|\calM_Q|/ n$ is an alternative way of writing $\cH_{t_2}^n(\bone_{G_Q}) $. The sum can be moved under the expectation in the last step because the events $\{j\in \calM_Q\}$ belong to $\calG_1$.

One of the major steps in the proof will be a derivation of a bound for 
$\E\left( f\left(H^{R}_{t_3}\right) \bone_{A_6^Q\cap A_7^Q}\right)$. Recall the definition  \eqref{m29.4}. The inequality below holds because  $A_6^Q\cap A_7^Q\subset A_2$ and we have \eqref{m31.1}.
\begin{align}\label{a21.1}
&\left(\wt\E_{t_2}(f_Q )- n^{-\alpha d +\gamma(-\alpha+3\delta/4) -\delta}
\right) 
\P(A_6^Q\cap A_7^Q) \\
&\qquad= \E\left(\left(\wt\E_{t_2}(f_Q )- n^{-\alpha d +\gamma(-\alpha+3\delta/4) -\delta}
\right) \bone_{A_6^Q\cap A_7^Q}\right)\notag\\
&\qquad\leq
\E\left(\cH_{t_2}^n(f_Q)\bone_{A_6^Q\cap A_7^Q}\right) 
= 
\frac 1n \E\left(\sum_{i\in\calM_Q} f\left(H^{i}_{t_2}\right)\bone_{A_6^Q\cap A_7^Q}\right). \notag 
\end{align}

If follows from the definition of $\calM_Q$ and the fact that $f$ depends only on the values of the process in $[0,t_1] $ that
\begin{align}\label{a21.3}
&
\sum_{i\in\calM_Q} f\left(H^{i}_{t_2}\right)\bone_{A_6^Q\cap A_7^Q}
=
\sum_{i\in\calM_Q} f\left(H^{i}_{t_3}\right)\bone_{A_6^Q\cap A_7^Q}.
\end{align}

Recall that $R$ and, therefore, $\bone_{A_7^Q}$ are $\calG_1$-measurable.
In the following calculation, we use \eqref{a21.3} in the first equality. The first inequality follows from \eqref{a24.1}. The second inequality follows from the fact that $A_6^Q\subset A_1$ and \eqref{m29.10}.
\begin{align}\notag 
&\frac 1n \E\left(
\sum_{i\in\calM_Q} f\left(H^{i}_{t_2}\right)\bone_{A_6^Q\cap A_7^Q}\right)
=\frac 1n \E\left(
\sum_{i\in\calM_Q} f\left(H^{i}_{t_3}\right)\bone_{A_6^Q\cap A_7^Q}\right)\\
&=\E\left(\frac 1n \E\left(
\sum_{i\in\calM_Q} f\left(H^{i}_{t_3}\right)\bone_{A_6^Q\cap A_7^Q}\mid \calG_1\right) \right)\notag \\
&\leq
\E\left( \E\left(
\cH_{t_2}^n(\bone_{G_Q}) f\left(H^{R}_{t_3}\right) \bone_{A_6^Q\cap A_7^Q}\mid \calG_1\right)\right)
\exp\left( 4\sqrt{d}  n^{ -\delta/4}\right) \notag \\
&\leq
\E\left( \E\left(
\left(\wt\P_{t_2}(G_Q) + n^{-\alpha d +\gamma(-\alpha+3\delta/4) -\delta}
\right)
 f\left(H^{R}_{t_3}\right) \bone_{A_6^Q\cap A_7^Q}\mid \calG_1\right)\right) 
 \exp\left( 4\sqrt{d}  n^{ -\delta/4}\right) \notag \\
&=
\left(\wt\P_{t_2}(G_Q) + n^{-\alpha d +\gamma(-\alpha+3\delta/4) -\delta}
\right)
 \E\left(
 f\left(H^{R}_{t_3}\right) \bone_{A_6^Q\cap A_7^Q}\right) 
 \exp\left( 4\sqrt{d}  n^{ -\delta/4}\right). \notag 
\end{align}
We combine this with \eqref{a21.1} to obtain
\begin{align*}
&\left(\wt\E_{t_2}(f_Q )- n^{-\alpha d +\gamma(-\alpha+3\delta/4) -\delta}
\right) 
\P(A_6^Q\cap A_7^Q)\\
&\qquad\leq
\left(\wt\P_{t_2}(G_Q) + n^{-\alpha d +\gamma(-\alpha+3\delta/4) -\delta}
\right)
\E\left(
 f\left(H^{R}_{t_3}\right) \bone_{A_6^Q\cap A_7^Q}\right) \exp\left( 4\sqrt{d}  n^{ -\delta/4}\right),
\end{align*}
and, therefore,
\begin{align}\label{a21.5}
\E\left(
 f\left(H^{R}_{t_3}\right) \bone_{A_6^Q\cap A_7^Q}\right)
&\geq
\frac{\left(\wt\E_{t_2}(f_Q )- n^{-\alpha d +\gamma(-\alpha+3\delta/4) -\delta}
\right) 
\P(A_6^Q\cap A_7^Q)}
{\left(\wt\P_{t_2}(G_Q) + n^{-\alpha d +\gamma(-\alpha+3\delta/4) -\delta}
\right)
\exp\left( 4\sqrt{d}  n^{ -\delta/4}\right)}.
\end{align}
By Lemma \ref{a11.6}, 
\begin{align*}
\wt\P_{t_2}\left(  G_Q \right)
\geq c_1 n^{-\alpha d +\gamma(-\alpha+3\delta/4) } ,
\qquad \wt\E_{t_2}\left(  f_Q \right)
\geq c_2 n^{-\alpha d +\gamma(-\alpha+3\delta/4) } ,
\end{align*}
where $c_2 = c_3 \wt\E_{t_2}\left(  f \right)$.
Thus  \eqref{a21.5} yields for large $n$,
\begin{align}\label{a24.3}
\E\left(
 f\left(H^{R}_{t_3}\right) \bone_{A_6^Q \cap A_7^Q}\right)
&\geq
\frac{\wt\E_{t_2}\left( f_Q\right)(1-n^{ -\delta/4} )\P(A_6^Q \cap A_7^Q)}
{\wt\P_{t_2}\left(  G_Q \right)(1+n^{ -\delta/4})
\exp\left( 4\sqrt{d}  n^{ -\delta/4}\right)}\\
&\geq 
\frac{\wt\E_{t_2}\left( f_Q\right)}
{\wt\P_{t_2}\left(  G_Q \right)} (1- n^{-\delta/8} )\P(A_6^Q \cap A_7^Q)
\notag\\
&= \wt\E_{t_2}\left( f_Q \mid B_{t_2}\in Q\right) (1- n^{-\delta/8} )\P(A_6^Q \cap A_7^Q)\notag \\
&=\wt\E_{t_2}\left( f \mid B_{t_2}\in Q\right) (1- n^{-\delta/8} )\P(A_6^Q \cap A_7^Q).
\notag
\end{align}
One can show in a similar manner that 
\begin{align}\label{rev29.2}
\E\left(
 f\left(H^{R}_{t_3}\right) \bone_{A_6^Q \cap A_7^Q}\right)
&\leq\wt\E_{t_2}\left( f \mid B_{t_2}\in Q\right) (1+ n^{-\delta/8} )\P(A_6^Q \cap A_7^Q).
\end{align}

By Lemma \ref{a11.1} and the assumption that $\|f\|_\infty \leq 1$, for every $Q\in \bQ$,
\begin{align*}\notag
&1-\eps\leq
\frac{
\wt\E_{t_2}\left( f \mid B_{t_2}\in Q\right)}
{\wt\E_{t_2}( f) } \leq 1 +\eps,\\
&\wt\E_{t_2}( f ) - \eps \leq \wt\E_{t_2}\left( f \mid B_{t_2}\in Q\right)
\leq \wt\E_{t_2}( f )  +\eps.
\end{align*}
We use this estimate, \eqref{a24.3} and \eqref{rev29.2} to see that
\begin{align}\label{rev2.1}
\E\left(
 f\left(H^{R}_{t_3}\right) \bone_{A_6^Q \cap A_7^Q}\right)
&\geq\left(\wt\E_{t_2}\left( f \right)-\eps\right) (1- n^{-\delta/8} )\P(A_6^Q \cap A_7^Q),\\
\E\left(
 f\left(H^{R}_{t_3}\right) \bone_{A_6^Q \cap A_7^Q}\right)
&\leq\left(\wt\E_{t_2}\left( f \right)+\eps\right) (1+ n^{-\delta/8} )\P(A_6^Q \cap A_7^Q).\label{rev2.2}
\end{align}

Recall that $R$ is an $ \calF_{t_3}^+$-measurable random variable  
taking values in  $\{X^1_{t_3}, \dots, X^n_{t_3}\}$ 
and $A_7^Q=\bigcup_{j\in \calM_Q}\{R=X^j_{t_3}\}$. 
Let $A_8=\bigcup_{Q\in\bQ} (A_6^Q\cap A_7^Q)$.
It follows from \eqref{rev2.1} that
\begin{align*}
\E\left(
 f\left(H^{R}_{t_3}\right) \bone_{A_8}\right)
&\geq \left(\wt\E_{t_2}\left( f \right)-\eps\right) (1- n^{-\delta/8} )\P( A_8) .
\end{align*}
This and   $\|f\|_\infty \leq 1$ imply that,
\begin{align*}
\E\left(
 f\left(H^{R}_{t_3}\right) \right) 
&\geq
\E\left(
 f\left(H^{R}_{t_3}\right) \bone_{A_8}\right)
\geq\wt\E_{t_2}\left( f \right) (1- n^{-\delta/8} )
\P(A_8)-\eps\\
&=\wt\E_{t_2}\left( f \right)\P(A_8)
-\wt\E_{t_2}\left( f \right) n^{-\delta/8} 
\P(A_8)-\eps\\
&\geq\wt\E_{t_2}\left( f \right)-
\wt\E_{t_2}\left( f \right)(1-\P(A_8))
- n^{-\delta/8} 
-\eps\\
&\geq\wt\E_{t_2}\left( f \right)
-\P(A_8^c)- n^{-\delta/8}-\eps.
\end{align*}

For the following upper bound, we use \eqref{rev2.2} and the fact that $\|f\|_\infty \leq 1$,
\begin{align*}
\E\left(
 f\left(H^{R}_{t_3}\right) \bone_{A_8}\right)
&\leq\left(\wt\E_{t_2}\left( f \right)+\eps\right) (1+ n^{-\delta/8} )\P( A_8). 
\end{align*}
This and $\|f\|_\infty \leq 1$ imply that
\begin{align*}
\E\left(
 f\left(H^{R}_{t_3}\right) \right) 
&\leq
\E\left(
 f\left(H^{R}_{t_3}\right) \bone_{A_8}\right)
+ \P(A_8^c)\\
&\leq\wt\E_{t_2}\left( f \right) (1+ n^{-\delta/8} )
\P(A_8)+ \P(A_8^c) +2\eps\\
&=\wt\E_{t_2}\left( f \right)\P(A_8)
+\wt\E_{t_2}\left( f \right) n^{-\delta/8} 
\P(A_8)+ \P(A_8^c) +2\eps\\
&\leq\wt\E_{t_2}\left( f \right)
+\P(A_8^c)+ n^{-\delta/8}+2\eps.
\end{align*}

The last two estimates show that
\begin{align}\label{a21.7}
\left|\E\left(
 f\left(H^{R}_{t_3}\right)\right)
-\wt\E_{t_2}\left( f \right)
\right| \leq \P(A_8^c)+ n^{-\delta/8}+2\eps .
\end{align}

Recall the definition of $\chi$ following the statement of Theorem \ref{thm:uniquespine} and note that $H^{\chi(t_3)}_{t_3}(t) = J^n(t)$ for $t\in[0,t_3]$. Hence we can apply \eqref{a21.7} to $R=\chi(t_3)$ to obtain,
\begin{align}\label{a1.1}
\left|\E\left(
 f\left(J^n\right)\right)
-\wt\E_{t_2}\left( f \right)
\right| \leq \P(A_8^c)+ n^{-\delta/8}+2\eps .
\end{align}

Recall definitions \eqref{sd.1} of $A_4$, \eqref{rev30.8} of $A^Q_6$
and \eqref{rev.m6.2} of $A^Q_7$. 
We have
\begin{align*}
A_8
=\bigcup_{Q\in\bQ} (A_6^Q\cap A_7^Q)
=\bigcup_{Q\in\bQ} (A_4\cap A_5^Q\cap A_7^Q)
= A_1\cap A_2\cap A_3 \cap \bigcup_{Q\in\bQ} ( A_5^Q
\cap   A_7^Q),
\end{align*}
so
\begin{align}\label{rev.m6.3}
A_8^c \subset A_1^c\cup A_2^c\cup A_3^c \cup \left(\bigcap_{Q\in\bQ}  A_5^Q
\right)^c
\cup
\left(
A_1\cap A_2\cap A_3 \cap \bigcap_{Q\in\bQ}  A_5^Q
\cap
\left(\bigcup_{Q\in\bQ}  A_7^Q
\right)^c \right).
\end{align}
In words, for the event $A_8$ to fail, one of the following must be true:

(i)
at least one of the following events fails: $A_1$, $A_2$, $A_3$,  $\bigcap_{Q\in\bQ}A_5^Q$, or 

(ii) the event $A_1\cap A_2\cap A_3 \cap \bigcap_{Q\in\bQ} A_5^Q$ holds and $\bigcup_{Q\in\bQ}A_7^Q$ fails.

By Lemmas \ref{a2.6} and \ref{a3.1}, 
\begin{align}\label{rev30.3}
\lim_{n\to\infty} \P\left(A_1\cap A_2\cap A_3 \cap \bigcap_{Q\in\bQ} A_5^Q\right)=1.
\end{align}

Suppose the event in \eqref{rev30.3} holds but $\bigcup_{Q\in\bQ}A_7^Q$ fails. Then  
$J^n_{t_3} = X^j_{t_3}$ for some $j \notin \bigcup_{Q\in\bQ} \calM_Q$. This implies, in view of \eqref{rev31.1} and \eqref{m29.5}, that
$\dist (X^j_{t_2} ,\Lambda^c) \leq 4n^{-\alpha +3\delta /4}$ for large $n$.
If $X^j$ did not jump in the interval $[t_2,t_3]$ then $J^n_{t_2} = X^j_{t_2}$ and $\dist (J^n_{t_2} ,\Lambda^c) \leq 4n^{-\alpha +3\delta /4}$. The last event is called $C^1_j$ in \eqref{rev30.4}, with $s_j$ playing the role of $t_2$. 
If $X^j$  jumps in the interval $[t_2,t_3]$ then the spine $J^n$ passes through a branch point in $[t_2,t_3]$. 
This event is called $C^2_j$ in \eqref{rev30.5}, with $[s_j,s_{j+1}]$ playing the role of $[t_2,t_3]$. 
It follows from \eqref{rev30.4}-\eqref{rev30.5} and Lemma \ref{a3.6} that
\begin{align*}
&\limsup_{n\to\infty} \P\left(A_1\cap A_2\cap A_3 \cap \bigcap_{Q\in\bQ} A_5^Q \cap\left( \bigcup_{Q\in\bQ} A_7^Q\right)^c\right)
\leq\eps.
\end{align*}
This,  \eqref{a1.1}, \eqref{rev.m6.3} and \eqref{rev30.3} imply that
\begin{align*}
\limsup_{n\to\infty} \left|\E\left(f\left(J^n\right) \right)
- \wt\E_{t_2}\left( f \right)\right|
&\leq 3\eps.
\end{align*}
By Lemma \ref{a14.3} and $\|f\|_\infty \leq 1$, 
\begin{align*}
\limsup_{n\to\infty} \left|\E\left(f\left(J^n\right) \right)
- \wh\E^\mu\left( f \right) \right|
&\leq 4\eps.
\end{align*}
Since $\eps>0$ is arbitrarily small, 
\begin{align*}
\lim_{n\to\infty} \E\left(f(J^n)  \right)
=  \wh\E^\mu\left( f \right).
\end{align*}
The function $f$ is an arbitrary  continuous non-negative function $f:C[0,\infty) \to \R$ with $\|f\|_\infty \leq 1$ and $\wt \E_{t_1} (f) >0$, which depends only on the values of the process on $[0,t_1]$, and $t_1>0$ is also arbitrary, so the theorem follows.
\end{proof}

\section{Estimates}\label{sec:est}

We will use notation presented in Section \ref{review} and in the proof of Theorem \ref{a1.7}.

\subsection{Large deviations and likelihood ratio}

\begin{lemma}\label{a2.5}

Suppose that $\alpha, \delta, t_2$ and $t_3$ are as in \eqref{a3.3}-\eqref{a7.8} and the paragraphs following these conditions.
For a process $B$ with values in $\R^d$, let
\begin{align*}
F(B)&=\left\{\sup_{s,t\in[t_2,t_3]} \left|B_{t} - B_{s}\right| < 
2 n^{-\alpha +3\delta /4}\right\}.
\end{align*}
Suppose that $B'$ and $B''$ are independent Brownian motions in $\R^d$.
If  $x^1,x^2 \in Q\in \bQ$ and $|x^j-z^k|  \leq 
n^{-\alpha +3\delta /4} $ for $j,k=1,2$ then
\begin{align*} 
&\frac{\P\left(\{B'_{t_3} \in dz^2,B''_{t_3} \in dz^1 ,  B'_{t_2} \in d x^1, B''_{t_2} \in d x^2\}\cap F(B')\cap F(B'')\right) }
{\P\left(\{B'_{t_3} \in dz^1, B''_{t_3} \in dz^2  ,  B'_{t_2} \in d x^1, B''_{t_2} \in d x^2\}\cap F(B')\cap F(B'')\right) } \leq
 \exp\left( 4\sqrt{d}  n^{ -\delta/4}\right).
\end{align*}
\end{lemma}

\begin{proof} 
\emph{Step 1}.
We will prove that if $B$ is Brownian motion in $\R^d$ then
there exists $c_1$ such that for any $x$ and $z$ such that $|x-z| \leq 
n^{-\alpha +3\delta /4}$ and sufficiently large $n$, the following bounds hold for the Radon-Nikodym derivative,
\begin{align}\label{rev22.1}
 1- \exp\left( - c_1 n^{\delta/2}\right)\leq
\frac {\P(B_{t_3} \in dz \mid F(B), B_{t_2} = x)}
{\P(B_{t_3} \in dz \mid  B_{t_2} = x) }
\leq 1+ \exp\left( - c_1 n^{\delta/2}\right).
\end{align}

In the first step, we will write $F$ instead of $F(B)$.
Standard estimates show that there exists $c_2$ such that for all $n\geq 2$,
\begin{align}\notag
\P(F^c)
&\leq\P\left(\sup_{s,t\in[t_2,t_3]} \left|B_{t} - B_{s}\right| \geq 
2 n^{-\alpha +3\delta /4}\right)
\leq \exp\left( - c_2
n^{2(-\alpha +3\delta /4)}/ (t_3-t_2)\right)\\
&\qquad= \exp\left( - c_2
n^{2(-\alpha +3\delta /4)}/ n^{-2\alpha + \delta}\right)
= \exp\left( - c_2 n^{\delta/2}\right).\label{a2.1}
\end{align}
Recall that $|x-z|  \leq  n^{-\alpha +3\delta /4} $.
Let $\calB(v,r)$ denote a ball with center $v$ and radius $r$.
 If $B_{t_2}=x$ then $\dist\left(z, \prt\calB\left(B_{t_2}, 2 n^{-\alpha +3\delta /4}\right)\right) \geq n^{-\alpha +3\delta /4} $.
Let $\tau$ be the hitting time of 
the sphere $\calS:=\prt\calB\left(B_{t_2}, 2 n^{-\alpha +3\delta /4}\right)$.
The last estimate, \eqref{a2.1} and an application of the strong Markov property at time $\tau$ imply that 
\begin{align}\label{a2.2}
&\P(B_{t_3} \in dz, F^c \mid B_{t_2} = x)
\leq \P(\tau <t_3  \mid B_{t_2} = x)
\sup_{y\in\calS, s\in[0,t_3-t_2]} \P(B_{s} \in dz \mid  B_{0} = y) \\
&\ \leq \P(F^c)  
\sup_{y\in\calS, s\in[0,t_3-t_2]} \P(B_{s} \in dz \mid  B_{0} = y)
\leq \exp\left( - c_2 n^{\delta/2}\right) 
\P(B_{t_3} \in dz \mid  B_{t_2} = x).\notag
\end{align}
We have
\begin{align}\label{oc6.1}
&\frac
{\P(B_{t_3} \in dz \mid F, B_{t_2} = x)}
{\P(B_{t_3} \in dz \mid  B_{t_2} = x)}
= \frac{\P(B_{t_3} \in dz , F\mid B_{t_2} = x)}
{\P(B_{t_3} \in dz \mid  B_{t_2} = x)\P( F\mid B_{t_2} = x)}\\
&\qquad= \frac{\P(B_{t_3} \in dz \mid B_{t_2} = x)
-\P(B_{t_3} \in dz , F^c\mid B_{t_2} = x)}
{\P(B_{t_3} \in dz \mid  B_{t_2} = x)\P( F\mid B_{t_2} = x)}\notag\\
&\qquad= \frac{1}
{\P( F\mid B_{t_2} = x)}\left(1
-\frac{\P(B_{t_3} \in dz , F^c\mid B_{t_2} = x)}
{\P(B_{t_3} \in dz \mid  B_{t_2} = x)}\right).\notag
\end{align}
This and \eqref{a2.2} yield
\begin{align}\label{oc6.2}
\frac {\P(B_{t_3} \in dz \mid F, B_{t_2} = x)}
{\P(B_{t_3} \in dz \mid  B_{t_2} = x) }
\geq
1
-\frac{\P(B_{t_3} \in dz , F^c\mid B_{t_2} = x)}
{\P(B_{t_3} \in dz \mid  B_{t_2} = x)}
\geq 1- \exp\left( - c_2 n^{\delta/2}\right).
\end{align}
It follows from
 \eqref{oc6.1} and \eqref{a2.1},  along with translation invariance, that for some $c_3$ and sufficiently large $n$,
\begin{align*}
\frac {\P(B_{t_3} \in dz \mid F, B_{t_2} = x)}
{\P(B_{t_3} \in dz \mid  B_{t_2} = x) }
&\leq \frac{1}
{\P( F\mid B_{t_2} = x)} =\frac{1}
{\P( F)} = \frac{1}
{1-\P( F^c)}\\
& \leq \frac{1}
{1-\exp\left( - c_2 n^{\delta/2}\right)}
\leq 1+ \exp\left( - c_3 n^{\delta/2}\right).
\end{align*}
This estimate and \eqref{oc6.2} imply \eqref{rev22.1}.

\emph{Step 2}.
We will prove that
f $B'$ and $B''$ are independent Brownian motions in $\R^d$,  $x^1,x^2 \in Q\in \bQ$ and $|x^j-z^k|  \leq 
n^{-\alpha +3\delta /4} $ for $j,k=1,2$ then
\begin{align} 
&\frac{\P\left(B'_{t_3} \in dz^2 , B'_{t_2} \in d x^1\right) }
{\P\left(B'_{t_3} \in dz^1,  B'_{t_2} \in d x^1\right) } 
\cdot
\frac{\P\left(B''_{t_3} \in dz^1 , B''_{t_2} \in d x^2\right) }
{\P\left( B''_{t_3} \in dz^2 , B''_{t_2} \in d x^2\right) } \leq
 \exp\left( 2\sqrt{d}  n^{ -\delta/4}\right).\label{rev22.2}
\end{align}

Suppose that $x^1,x^2 \in Q\in \bQ$. Then $|x^1-x^2| \leq \sqrt{d} n^{-\alpha}$. 
If in addition $|x^j-z^k|  \leq 
n^{-\alpha +3\delta /4} $ for $j,k=1,2$ then
\begin{align*} \notag 
&\frac{\P\left(B'_{t_3} \in dz^2 , B'_{t_2} \in d x^1\right) }
{\P\left(B'_{t_3} \in dz^1,  B'_{t_2} \in d x^1\right) } 
\cdot
\frac{\P\left(B''_{t_3} \in dz^1 , B''_{t_2} \in d x^2\right) }
{\P\left( B''_{t_3} \in dz^2 , B''_{t_2} \in d x^2\right) }\\
&\qquad=\frac{\P(B'_{t_3} \in dz^2,B''_{t_3} \in dz^1 \mid  B'_{t_2} = x^1, B''_{t_2} = x^2) }
{\P(B'_{t_3} \in dz^1, B''_{t_3} \in dz^2 \mid  B'_{t_2} = x^1, B''_{t_2} = x^2) }(z^1,z^2) \notag \\
&\qquad=\frac
{\exp\left( -|x^2-z^1|^2/(2(t_3-t_2)) \right)}
{\exp\left( -|x^1-z^1|^2/(2(t_3-t_2)) \right)}
\cdot
\frac
{\exp\left( -|x^1-z^2|^2/(2(t_3-t_2)) \right)}
{\exp\left( -|x^2-z^2|^2/(2(t_3-t_2)) \right)} \notag \\
&\qquad\leq \frac
{\exp\left( -(|x^2-x^1|-|x^1-z^1|)^2/(2(t_3-t_2)) \right)}
{\exp\left( -|x^1-z^1|^2/(2(t_3-t_2)) \right)} \notag \\
&\qquad\qquad \times
\frac
{\exp\left( -(|x^1-x^2|-|x^2-z^2|)^2/(2(t_3-t_2)) \right)}
{\exp\left( -|x^2-z^2|^2/(2(t_3-t_2)) \right)}
 \notag \\
&\qquad= \exp\left( -(|x^2-x^1|^2-2|x^2-x^1|\cdot|x^1-z^1|)/(2(t_3-t_2)) \right) \notag \\
&\qquad\qquad \times
\exp\left( -(|x^1-x^2|^2-2|x^2-x^1|\cdot|x^2-z^2|)/(2(t_3-t_2)) \right)
 \notag \\
&\qquad\leq \exp\left( |x^2-x^1|\cdot|x^1-z^1|)/(t_3-t_2) \right)
\exp\left( |x^2-x^1|\cdot|x^2-z^2|)/(t_3-t_2) \right) \notag \\
&\qquad \leq  \exp\left(2\sqrt{d} n^{-\alpha} n^{-\alpha+3\delta/4} /  n^{-2\alpha +\delta}\right)
= \exp\left( 2\sqrt{d}  n^{ -\delta/4}\right).
\end{align*}

\emph{Step 3}.
We use independence of $B'$ and $B''$, \eqref{rev22.1} and \eqref{rev22.2} to obtain for large $n$,
\begin{align*} 
&\frac{\P\left(\{B'_{t_3} \in dz^2,B''_{t_3} \in dz^1 ,  B'_{t_2} \in d x^1, B''_{t_2} \in d x^2\}\cap F(B')\cap F(B'')\right) }
{\P\left(\{B'_{t_3} \in dz^1, B''_{t_3} \in dz^2  ,  B'_{t_2} \in d x^1, B''_{t_2} \in d x^2\}\cap F(B')\cap F(B'')\right) } \\
&=\frac{\P\left(\{B'_{t_3} \in dz^2 , B'_{t_2} \in d x^1\}\cap F(B')\right) }
{\P\left(\{B'_{t_3} \in dz^1,  B'_{t_2} \in d x^1\}\cap F(B')\right) } 
\cdot
\frac{\P\left(\{B''_{t_3} \in dz^1 , B''_{t_2} \in d x^2\}\cap F(B'')\right) }
{\P\left(\{ B''_{t_3} \in dz^2 , B''_{t_2} \in d x^2\}\cap F(B'')\right) }\\
&\leq
\frac{\P\left(B'_{t_3} \in dz^2 , B'_{t_2} \in d x^1\right) }
{\P\left(B'_{t_3} \in dz^1,  B'_{t_2} \in d x^1\right) } 
\cdot
\frac{\P\left(B''_{t_3} \in dz^1 , B''_{t_2} \in d x^2\right) }
{\P\left( B''_{t_3} \in dz^2 , B''_{t_2} \in d x^2\right) }
\cdot \frac{\left(1+ \exp\left( - c_3 n^{\delta/2}\right)\right)^2}
{\left(1- \exp\left( - c_3 n^{\delta/2}\right)\right)^2}
\\
&\leq
\exp\left( 2\sqrt{d}  n^{ -\delta/4}\right)
 \frac{\left(1+ \exp\left( - c_3 n^{\delta/2}\right)\right)^2}
{\left(1- \exp\left( - c_3 n^{\delta/2}\right)\right)^2}
\leq
 \exp\left( 4\sqrt{d}  n^{ -\delta/4}\right).
\end{align*}
\end{proof}

Recall  definitions of $A_3, \Lambda''(Q)$, $A_5^Q(z^1, \dots,z^N)$ and $A_5^Q$  stated in \eqref{a2.3}, \eqref{rev30.2}, \eqref{rev.m6.1} and \eqref{rev30.7}. 

\begin{lemma}\label{a2.6}
(i)
$\lim_{n\to\infty} \P(A_3) = 1$.

(ii) 
\begin{align*}
\lim_{n\to\infty} \P\left(\bigcap_{Q\in\bQ}  A_5^Q\right)=1.
\end{align*}
\end{lemma}

\begin{proof}
(i) We use \eqref{a2.1} to see that
\begin{align}\label{a2.4}
\P(A_3^c) &\leq \sum_{1\leq j \leq n}  \P
\left(\sup_{s,t\in[t_2,t_3]}\left |B^{j}_s-B^{j}_t\right| \geq 
2 n^{-\alpha +3\delta /4}\right)\leq n \exp\left( - c_1 n^{\delta/2}\right).
\end{align}
Since $\delta>0$, (i) follows.

(ii)
The number of 
 $Q\in\bQ$ is bounded by $ n^{\alpha d}|\Lambda|$.
Note that   \eqref{a2.1} applies (with a different constant)  when $2 n^{-\alpha +3\delta /4}$ is replaced with $(1/2) n^{-\alpha +3\delta /4}$. Thus
\begin{align*}
\P\left(\left(\bigcap_{Q\in\bQ}  A_5^Q\right)^c\right)
&= \P\left(\bigcup_{Q\in\bQ} \left( \bigcup _{(z^1, \dots,z^N) \in \Lambda''(Q)} A_5^Q(z^1, \dots,z^N)\right)^c\;\right)\\
&\leq  n^{\alpha d}|\Lambda| n \exp\left( - c_2 n^{\delta/2}\right).
\end{align*}
Part (ii) follows since $\delta>0$.
\end{proof}

\subsection{Conditioned space-time Brownian motion}

Let $\P_{x,y,s}$ denote the distribution of $\{B_t, 0 \leq t \leq s\}$ where $B$ is Brownian motion  starting from $B_0=x$, conditioned by $\{B_s=y\}$, and conditioned to stay inside $\Lambda$ on the interval $[0,s]$. The distribution  
$\P_{x,y,s}$ can be thought of as the distribution of the space component of the space-time Brownian motion $(B_t,t)$ conditioned by the parabolic function $h$ in $\Lambda\times [0,s]$, equal to 0 everywhere on the boundary  of $\Lambda\times [0,s]$ except for $(y,s)$. Such processes are known as $h$-processes.

Let $p_t(x,y)$ denote the transition density for Brownian motion killed upon exiting $\Lambda$.
Let $\lambda>0$ be the first eigenvalue for  $(-\frac 1 2) \Delta$ with Dirichlet boundary conditions in $\Lambda$ and let $\vphi>0$ be the corresponding eigenfunction.
A bounded Lipschitz domain is intrinsically ultracontractive (IU). We will need only one result on IU domains, cited below, so we will not define IU domains here; instead we ask the reader to consult, e.g., 
\cite{DS,banu}. It follows from \cite[(1.8)]{banu} that  for any $\eta>0$ there exists $u$ such that for $s\geq u$ and  $x,y\in \Lambda$,
\begin{align}\label{a14.2}
1-\eta \leq
\frac{ p_{s}(x,y)}{e^{-\lambda s} \vphi(x) \vphi(y)}\leq 1+\eta.
\end{align}

\begin{lemma}\label{a11.4}
Let $\wt p_t(x,y)$ denote the transition density for Brownian motion conditioned to stay in $\Lambda$ on $[0,t]$.
There exists $s_1$ and $c_1,c_2,\gamma\in(0,\infty)$ such that for all $t\geq s_1$ and $x,y\in \Lambda$,
\begin{align}\label{a13.7}
c_1 \dist(y, \Lambda^c)^\gamma \leq
\wt p_t(x,y) \leq c_2 .
\end{align}
\end{lemma}

Results of this type are well known but we could not find an exact reference for \eqref{a13.7}.

\begin{proof}[Proof of Lemma \ref{a11.4}]

Since $\Lambda$ is a bounded Lipschitz domain with the Lipschitz constant less than 1, we can find $\rho>0$, a finite number $k_1$ of 
points $x^k\in \prt \Lambda$, orthonormal coordinate systems $CS_k$ and  Lipschitz functions $\psi_k:\R^{d-1}\to\R$ with the Lipschitz constant less than 1, satisfying the following conditions.
For $y\in \prt \Lambda$ and $r>0$, let 
\begin{align*}
\Gamma_k(y,r) = \{y+(x_1,\dots,x_d): x_1^2+\dots+x_{d-1}^2 \leq r^2, x_d^2\leq r^2\}, \qquad \text{  in  } CS_k.
\end{align*}
Moreover, in $CS_k$,
\begin{align*}
&\Lambda\cap \Gamma_k(x^k,\rho) = \{(x_1,\dots,x_d)\in\Gamma_k(x^k,\rho): x_d > \psi_k(x_1,\dots,x_{d-1})\}, \qquad k=1,\dots,k_1,\\
&\{x\in\Lambda: \dist(x, \Lambda^c) \leq \rho/4\} \subset \bigcup_{1\leq k\leq k_1} \Gamma_k(x^k,\rho/2).
\end{align*}

Fix $CS_k$ and $y\in \prt \Lambda \cap \Gamma_k(x^k, \rho/2)$.
Let $\n_k = (0,\dots,0,1)$ in $CS_k$.
Let $L_y=\{y+a \n_k: a> 0\}$ and let $\calC_{y,\alpha}$ be the open cone consisting of all half-lines with the endpoint at $y$ forming the angle less than $\alpha$ with $L_y$. Note that
\begin{align*}\notag
&\Gamma_k(y,\rho/2) \cap \Lambda \subset \Gamma_k(x^k,\rho)\cap \Lambda,\\
& \calC_{y,\pi/8} \cap \Gamma_k(y,\rho/2) \subset
\Lambda\cap \Gamma_k(y,\rho/2).
\end{align*}

It is well known (see, e.g., ``Application'' on page 192 in \cite{Burk}) that
there exists a positive harmonic function $h_1(x)$, $x\in \calC_{y,\pi/8}$, such that $h_1(x)=0$ for $x\in \prt \calC_{y,\pi/8}$ and $h_1(x) = |x-y|^\gamma f(\theta)$ for some $\gamma>0$ and a function $f$, where $\theta$ is the angle between the line segment $\ol {0,x}$ and $L_y$. We can and will assume that $f(0)=1$.

Let $h_2(x)$ be a positive harmonic function in $\Lambda\cap \Gamma_k(y,\rho/2)$, with the boundary values $h_2(x) = h_1(x)$ for $x\in \calC_{y,\pi/8}\cap \prt \Gamma_k(y,\rho/2)\cap \Lambda$ and $h_2(x) = 0$ for $x\in 
\prt(\Lambda\cap \Gamma_k(y,\rho/2)) \setminus \calC_{y,\pi/8}$.
Then $h_2(x) \geq h_1(x)$ for $ x \in \prt(\calC_{y,\pi/8}\cap  \Gamma_k(y,\rho/2))$. It follows by the
elliptic maximum principle that $h_2(x) \geq h_1(x)$ for $ x \in L_y \cap \Lambda\cap \Gamma_k(y,\rho/2)$. Hence,
\begin{align}\label{a13.1}
&h_2(x) \geq |x-y|^\gamma, \qquad x\in L_y,\\
&h_2(x) \leq c_3 :=\sup \left\{h_1(x): x\in \calC_{y,\pi/8}, |x-y| \leq 2\rho\right\}, \qquad x\in \Lambda\cap \Gamma_k(y,\rho/2).\label{a13.2}
\end{align}

Recall that $p_t(x,z)$ denotes the transition density for Brownian motion killed upon exiting $\Lambda$. Let $g(t,x) = h_2(x)$ for $x\in \Lambda\cap \Gamma_k(y,\rho/2)$ and $t\geq 0$. The functions $(z,t)\to p_t(x,z)$ and 
$(z,t)\to g(t,z) $ are parabolic, i.e.,  they are solutions to the heat equation and they have zero boundary values on $(\prt \Lambda\cap \Gamma_k(y,\rho/2))\times \R$. Hence, by  \cite[Thm. 1.6]{FGS}, for $s>0$, there exist $c_4$ and $a>b_1>0$ depending only on $\Lambda$ and $s$, such that for $t\geq s$, $0<b<b_1$ and $x\in \Lambda$,
\begin{align*}
\frac{g(t,y+b\n_k)}{p_t(x,y+b\n_k)} 
\leq c_4\frac{g(t+2a^2,y+a\n_k)}{p_{t-2a^2}(x,y+a\n_k)},
\end{align*}
so, using \eqref{a13.1},
\begin{align}\notag
&p_t(x,y+b\n_k) \geq c_4^{-1}
\frac{g(t,y+b\n_k)}{g(t+2a^2,y+a\n_k)}
p_{t-2a^2}(x,y+a\n_k)\\
&\qquad= c_4^{-1}
\frac{h_2(y+b\n_k)}{h_2(y+a\n_k)}
p_{t-2a^2}(x,y+a\n_k)
\geq c_4^{-1}
\frac{b^\gamma}{h_2(y+a\n_k)}
p_{t-2a^2}(x,y+a\n_k).\label{a13.3}
\end{align}

Let $\Lambda_a^c$ be the set of all points of the form $y+ a_1\n_k$ with $0<a_1<a$, where $y$ can be any point in $ \prt \Lambda \cap \Gamma_k(x^k, \rho/2)$ and $k$ can be any $1,\dots, k_1$. Let $\Lambda_a=\Lambda \setminus \Lambda_a^c$.

By \eqref{a14.2}, for any $\eta>0$, $a$ as above, sufficiently large $s_1$ and $t\geq s_1$,
\begin{align}\label{a13.4}
&p_{t-2a^2}(x,y+a\n_k) \geq  p_{t}(x,y+a\n_k)/2,\\
&\inf_{z\in \Lambda_a} p_{t}(x,z) \geq 
(1-\eta)e^{-\lambda t} \vphi(x) \inf_{z\in \Lambda_a}\vphi(z).\label{a13.5}
\end{align}
Since $\dist(\Lambda_a, \Lambda^c) >0$, $c_5:=\inf_{z\in \Lambda_a}\vphi(z)>0$.
This observation, \eqref{a13.2}, \eqref{a13.3}, \eqref{a13.4} and \eqref{a13.5} yield
\begin{align}\label{a13.6}
p_t(x,y+b\n_k) 
\geq c_4^{-1}
\frac{b^\gamma}{c_3}(1-\eta)e^{-\lambda t} \vphi(x) c_5
= c_6 b^\gamma e^{-\lambda t} \vphi(x).
\end{align}

By \eqref{a14.2}, for sufficiently large $t$, 
\begin{align*}
\int_\Lambda p_t(x,z)dz \leq 2 e^{-\lambda t} \vphi(x)\int_\Lambda \vphi(z)dz = c_7 e^{-\lambda t} \vphi(x) .
\end{align*}
It follows from this and \eqref{a13.6} that
\begin{align*}
\wt p_t(x,y+b\n_k)=
\frac{p_t(x,y+b\n_k)}
{\int_\Lambda p_t(x,z)dz} 
\geq 
\frac{ c_6 b^\gamma e^{-\lambda t} \vphi(x)}
{c_7 e^{-\lambda t} \vphi(x)} = c_8 b^\gamma.
\end{align*}
This proves the lower bound in \eqref{a13.7}
for $y\in \Lambda_a^c$ because $b \geq \dist(y+b\n_k,\Lambda^c) $. 
We can make the bound valid for all $y\in \Lambda$ by making  $c_1>0$ smaller, if necessary.

Since 
\begin{align*}
 \wt p_{t}(x,y) = \frac { p_{t}(x,y)}{\int_{\Lambda} p_{t}(x,y)dy},
\end{align*}
it follows from \eqref{a14.2} that  for any $\eta_1>0$ there exists a $u>0$ 
such that for $s\geq u$ and  $x,y\in \Lambda$,
\begin{align}\label{rev30.11}
\left(\int_{\Lambda}\vphi(x)\, dx\right)^{-1}(1-\eta_1) \leq
\frac{\wt p_{s}(x,y)}{ \vphi(y) }\leq \left(\int_{\Lambda}\vphi(x)\, dx\right)^{-1}(1+\eta_1).
\end{align}
The upper bound in \eqref{a13.7} follows because $\sup_{y\in\Lambda} \vphi(y) <\infty$.
\end{proof}

\begin{remark}\label{rev30.10}
Suppose that  $ \P'$ and $ \P'' $ are probability measures  on space-time trajectories $\{(B_s,s), 0\leq s \leq t\}$ such that under each of these measures,
the process $\{B_s,0\leq s \leq t\}$ is the space component of a space-time  Brownian motion conditioned to exit $\Lambda\times (0,t)$ via $\Lambda \times \{t\}$, but with different exit distributions. 
It follows from the theory of $h$-processes (i.e., conditioned Brownian motion; see \cite{Doob})
that under both $ \P'$ and $ \P'' $,
the process $\{B_s,0\leq s \leq t\}$ conditioned by $\{B_0=x, B_{t} = y\}$ has the same distribution.
It follows that for $u\leq t$, the Radon-Nikodym derivative $d\P'/d\P''$ on the $\sigma$-field $\sigma(B_s,0\leq s \leq u)$  depends only on $B_u$. 
\end{remark}

Recall that for $x,y\in\Lambda$, $\P_{x,y,s}$ denotes the distribution of $\{B_t, 0 \leq t \leq s\}$ where $B$ is Brownian motion  starting from $B_0=x$, conditioned  to stay inside $\Lambda$ on the interval $[0,s]$ and further conditioned by $\{B_s=y\}$.

\begin{lemma}\label{a11.1}
Fix any $t_1,\eps>0$. There exists
 $s_1>t_1$  so large that for all $x,y_1,y_2\in \Lambda$ and $t_2\geq s_1$,
\begin{align}\label{m28.1}
1-\eps\leq
\frac{d \P_{x,y_1,t_2}}{d \P_{x,y_2,t_2}} \leq 1+\eps
\end{align}
on the $\sigma$-field $\calF^B_{t_1}:=\sigma(B_t, 0\leq t \leq t_1)$.
\end{lemma}

\begin{proof}

Recall that 
$\wt p_t(x,y)$ denotes the transition density for Brownian motion conditioned to stay in $\Lambda$ on $[0,t]$.
By the Markov property applied at $t_1$,
the density of $B_{t_1}$ under $\P_{x,y_1,t_2}$ evaluated at $z$ is
\begin{align*}
\frac{ \wt p_{t_1}(x, z) \wt p_{t_2-t_1}(z,y_1)}
{\wt p _{t_2}(x,y_1)},
\end{align*}
and the analogous formula holds for $\P_{x,y_2,t_2}$.
This and Remark \ref{rev30.10} imply that
\begin{align}\label{a11.3}
\left. \frac{d \P_{x,y_1,t_2}}{d \P_{x,y_2,t_2}} \right|
_{\calF^B_{t_1}}
&=
\frac{ \wt p_{t_1}(x, z) \wt p_{t_2-t_1}(z,y_1)}
{\wt p _{t_2}(x,y_1)}
\cdot
\frac
{\wt p _{t_2}(x,y_2)}
{ \wt p_{t_1}(x, z) \wt p_{t_2-t_1}(z,y_2)}\\
&=
\frac{  \wt p_{t_2-t_1}(z,y_1)\wt p _{t_2}(x,y_2)}
{\wt p _{t_2}(x,y_1)\wt p_{t_2-t_1}(z,y_2)}
.\notag
\end{align}
If we assume that $B$ is defined on the canonical probability space then we can identify $\{B_t, 0\leq t \leq t_2\}$ with $\{\omega_t, 0\leq t \leq t_2\}$, where $\omega \in C([0,t_2], \Lambda)$, and rewrite \eqref{a11.3} as
\begin{align*}
\left. \frac{d \P_{x,y_1,t_2}}{d \P_{x,y_2,t_2}} \right|
_{\calF^B_{t_1}}(\omega)
=
\frac{  \wt p_{t_2-t_1}(\omega_{t_1},y_1)\wt p _{t_2}(x,y_2)}
{\wt p _{t_2}(x,y_1)\wt p_{t_2-t_1}(\omega_{t_1},y_2)}.
\end{align*}
Combining this with \eqref{rev30.11}  completes the proof.
\end{proof}

\begin{lemma}\label{a11.6}
Recall from \eqref{m29.3} that $\wt\P_{t_2}$ refers to the Wiener measure conditioned on
staying inside $\Lambda$ up to $t_2$.
We have assumed  that $f$ depends only on the values of the process in $[0,t_1]$ and $\wt \E_{t_2} (f) >0$. 
Assume that $t_2 - t_1 \geq s_1$ where $s_1$ is given in Lemma \ref{a11.4}.
Then there exist $c_1$ and $c_2$ such that for every $Q\in\bQ$,
\begin{align}\label{a11.7}
c_1 n^{-\alpha d+ \gamma (-\alpha+3\delta/4)}
&\leq	\wt\P_{t_2}(G_Q )
\leq c_2 n^{-\alpha d},\\
c_1 \wt \E_{t_2} (f) n^{-\alpha d+ \gamma (-\alpha+3\delta/4)}
&\leq	\wt\E_{t_2}(f_Q ) \leq c_2 \wt \E_{t_2} (f) n^{-\alpha d}.
\label{a11.8}
\end{align}
\end{lemma}

\begin{proof}
We will apply Remark \ref{rev30.10}.
The volume of $Q$ is $n^{-\alpha d}$. Thus \eqref{a11.7} follows from the assumption \eqref{m29.5} and Lemma \ref{a11.4}.

Let $\E_{x,y,t}$ denote the expectation corresponding to the probability measure $\P_{x,y,t}$ discussed in Lemma \ref{a11.1}. Using that lemma, for an arbitrary fixed $z\in \Lambda$,
\begin{align*}
&\wt\E_{t_2}(f_Q ) =
\int_Q \int_{\Lambda_1} \E_{x,y,t_2}(f) d\mu_n(dx) \wt \P_{t_2}(dy)
\leq \int_Q \int_{\Lambda_1}(1+\eps)\E_{x,z,t_2}(f) d\mu_n(dx) \wt \P_{t_2}(dy)
\\
&= (1+\eps)\wt\P_{t_2}(G_Q )
\int_{\Lambda_1}\E_{x,z,t_2}(f) d\mu_n(dx)  \\
&= (1+\eps)\wt\P_{t_2}(G_Q )
\int_\Lambda \int_{\Lambda_1}\E_{x,z,t_2}(f) d\mu_n(dx)  \wt \P_{t_2}(dy)\\
&\leq (1+\eps)^2\wt\P_{t_2}(G_Q )
\int_\Lambda \int_{\Lambda_1}\E_{x,y,t_2}(f) d\mu_n(dx)  \wt \P_{t_2}(dy)
=(1+\eps)^2\wt\P_{t_2}(G_Q )
\wt\E_{t_2}(f ) .
\end{align*}
Similarly,
\begin{align*}
&\wt\E_{t_2}(f_Q ) 
\geq(1-\eps)^2\wt\P_{t_2}(G_Q )
\wt\E_{t_2}(f ) .
\end{align*}
Now \eqref{a11.8} follows from  \eqref{a11.7}.
\end{proof}

Recall that $\mu$ and $\mu_n$ are probability measures supported in  $\Lambda_1\subset \Lambda$ and  $\mu_n$  converge weakly to $\mu$ as $n\to \infty$.
The next lemma follows essentially from the main theorem in \cite{Pin} but we need a specific order of quantifiers which does not seem to follow directly from that theorem. Recall that $\wh \P^\mu$  denotes the distribution of Brownian motion  conditioned to stay in $\Lambda$ forever, with the initial distribution 
$c \vphi(x)\mu(dx)$, where $c>0$ is the normalizing constant.

\begin{lemma} \label{a14.3}
For every $\eps>0$ and $t_1>0$ there exists $s_1$ such that for every positive continuous function $f$ on $C[0,t_1]$ with $\|f\|_\infty \leq 1$ and $t\geq s_1$,
\begin{align*}
1-\eps \leq 
\liminf_{n\to\infty}
\frac{\wt \E_t^{\mu_n}(f)}{\wh \E^\mu(f)} \leq
\limsup_{n\to\infty}
\frac{\wt \E_t^{\mu_n}(f)}{\wh \E^\mu(f)} \leq 1+\eps.
\end{align*}
\end{lemma}

\begin{proof}
Let $\vphi>0$ denote the first eigenfunction of  $(-\frac 1 2) \Delta$ in $\Lambda$ with Dirichlet boundary conditions and let $\lambda>0$ be the corresponding eigenvalue.
The function $h(x,t)= e^{\lambda t} \vphi(x)$ is parabolic in $\Lambda\times (0,\infty)$.
The transition density $\wh p_t(x,y)$ for Brownian motion conditioned not to exit $\Lambda$  is given by, for $t>0$ and $x,y\in \Lambda$,
\begin{align}\label{a13.9}
\wh p_t(x,y) = \frac 1 {h(x,0)} p_t(x,y) h(y,t)
=e^{\lambda t} \frac{\vphi(y)}{\vphi(x)}p_t(x,y).
\end{align}

Under both $\wt \P_t$ and $\wh \P $,
the process $\{B_s,0\leq s \leq t_1\}$ is the space component of a space-time  Brownian motion conditioned to exit $\Lambda\times (0,t_1)$ via $\Lambda \times \{t_1\}$, but with different exit distributions. 
We will show that for every $\eps>0$ there exists $s_1$ such that for $t\geq s_1$ and $x,y,v,z\in \Lambda$,
\begin{align}\label{a14.1}
1-\eps \leq
\frac{\wt \P_t(B_{t_1}\in dy \mid B_0=x, B_t = v)}{\wh \P (B_{t_1} \in dy \mid B_0=x, B_t=z)}
\leq 1+\eps.
\end{align}
By \eqref{a13.9},
\begin{align}\label{rev20.2}
\wt \P_t(B_{t_1}\in dy \mid B_0=x, B_t = v)
&= \frac { p_{t_1}(x,y) p_{t-t_1}(y,v)}{ p_{t}(x,v)}dy,\\
\wh \P (B_{t_1}\in dy \mid B_0=x, B_t = z)
&= \frac {\wh p_{t_1}(x,y)\wh p_{t-t_1}(y,z)}{\wh p_{t}(x,z)}dy \notag \\
&= \frac {e^{\lambda t_1} \frac{\vphi(y)}{\vphi(x)} p_{t_1}(x,y) 
e^{\lambda (t-t_1)} \frac{\vphi(z)}{\vphi(y)}p_{t-t_1}(y,z)}
{e^{\lambda t} \frac{\vphi(z)}{\vphi(x)} p_{t}(x,z)}dy \notag \\
&= \frac { p_{t_1}(x,y) 
p_{t-t_1}(y,z)}{ p_{t}(x,z)}dy.\notag
\end{align}
Hence,
\begin{align*}
\frac{\wt \P_t(B_{t_1}\in dy \mid B_0=x, B_t = v)}{\wh \P (B_{t_1} \in dy \mid B_0=x, B_t=z)}
=\frac {  p_{t-t_1}(y,v) p_{t}(x,z)}{ p_{t}(x,v)p_{t-t_1}(y,z)}.
\end{align*}

By \eqref{a14.2}, for any $\eta>0$ there exists $s_2$ such that for $t-t_1\geq s_2$  and $x,y,v,z\in \Lambda$,
\begin{align*}
\frac{\wt \P_t(B_{t_1}\in dy \mid B_0=x, B_t = v)}{\wh \P (B_{t_1} \in dy \mid B_0=x, B_t=z)}
\leq
\frac{(1+\eta)^2}{(1-\eta)^2}
\frac { 
e^{-\lambda(t-t_1)} \vphi(y) \vphi(v)
e^{-\lambda t} \vphi(x) \vphi(z)
}{ 
e^{-\lambda t} \vphi(x) \vphi(v)
e^{-\lambda(t-t_1)} \vphi(y) \vphi(z)}
=\frac{(1+\eta)^2}{(1-\eta)^2} .
\end{align*}
This proves the upper bound in \eqref{a14.1}. The lower bound can be proved in an analogous way.

Suppose that $f$ is a positive continuous function  on $C[0,t_1]$ with $\|f\|_\infty \leq 1$.
Recall from Remark \ref{rev30.10}
that under both $\wt \P_t$ and $\wh \P $,
the process $\{B_s,0\leq s \leq t_1\}$ conditioned by $\{B_0=x, B_{t_1} = y\}$ has the same distribution. Hence, \eqref{a14.1} implies that 
 for every $\eps>0$ there exists $s_1$ such that for $t\geq s_1$ and $x,v,z\in \Lambda$,
\begin{align*}
1-\eps \leq
\frac{\wt \E_t(f \mid B_0=x, B_t = v)}{\wh \E (f \mid B_0=x, B_t=z)}
\leq 1+\eps.
\end{align*}
Since $v,z\in \Lambda$ are arbitrary, for every $\eps>0$ there exists $s_1$ such that for $t\geq s_1$ and $x\in \Lambda$,
\begin{align}\label{rev31.2}
1-\eps \leq
\frac{\wt \E_t(f \mid B_0=x)}{\wh \E (f \mid B_0=x)}
\leq 1+\eps.
\end{align}

By \eqref{a14.2}, for any $\eta>0$, sufficiently large $t$,  normalizing constants $c_n$ and $c_*$,
\begin{align}\label{rev31.4}
\wt\P_t(B_0\in dx) &= c_n \mu_n(dx)\int_\Lambda p_t(x, y) dy
\geq (1-\eta)c_n \mu_n(dx)\int_\Lambda e^{-\lambda t} \vphi(x) \vphi(y) dy\\
&= c_n c_* (1-\eta)\vphi(x) \mu_n(dx),\notag
\end{align}
and similarly 
\begin{align}\label{rev31.5}
\wt\P_t(B_0\in dx) \leq  c_nc_* (1+\eta)\vphi(x) \mu_n(dx).
\end{align}
Recall that $\mu_n$ converge weakly to $\mu$ and  $\wh \P(B_0\in dx)=c \vphi(x) \mu(dx)$ for a normalizing constant $c$. It follows from \eqref{rev31.5} that
\begin{align*}
1&\leq\liminf_{n\to\infty}\int c_nc_* (1+\eta)\vphi(x) \mu_n(dx)
=
\liminf_{n\to\infty}\int \frac{c_nc_* (1+\eta)}{c} c\vphi(x) \mu_n(dx)\\
&=
\liminf_{n\to\infty}\int \frac{c_nc_* (1+\eta)}{c} c\vphi(x) \mu(dx)
=\liminf_{n\to\infty} \frac{c_nc_* (1+\eta)}{c},
\end{align*}
so
\begin{align*}
\liminf_{n\to\infty} c_n \geq c/(c_*(1+\eta)).
\end{align*}
This and \eqref{rev31.4} imply that
\begin{align}\label{rev31.12}
\liminf_{n\to\infty} \wt\P_t&(B_0\in dx) \geq   \liminf_{n\to\infty}
c_n c_* (1-\eta)\vphi(x) \mu_n(dx)\\
& \geq   \liminf_{n\to\infty}
\frac{1-\eta}{1+\eta} c\vphi(x) \mu_n(dx)
=\frac{1-\eta}{1+\eta} \wh \P(B_0 \in dx) .\notag
\end{align}
Similarly,
\begin{align}\label{rev31.11}
\limsup_{n\to\infty} \wt\P_t&(B_0\in dx) \leq  
\frac{1+\eta}{1-\eta} \wh \P(B_0 \in dx) .
\end{align}

We have assumed in
Theorem \ref{a1.7} that $\mu$ is  supported in a set $\Lambda_1\subset \Lambda$ such that $\dist(\Lambda_1, \Lambda^c) >0$.
It follows from \cite[Cor.~1]{CB} that if $x^k \in \Lambda_1$ and $x^k\to x^\infty\in\Lambda_1$ then  $\wt \E_t(f\mid B_0=x^k) \to\wt \E_t(f\mid B_0=x^\infty) $. 
A standard coupling argument can be applied to construct processes $B^n$ with the same transition probabilities as those of $B$, on the same probability space, with the initial distributions $\mu_n$ for $B^n$  and $\mu$ for $B$, and such that $|B^n_0 - B_0|\to0$, a.s., as $n\to \infty$. 
These observations, \eqref{rev31.2} and \eqref{rev31.11} imply that
\begin{align*}
&\limsup_{n\to\infty}\wt \E_t(f )
=\limsup_{n\to\infty}
 \int \wt \E_t(f \mid B_0=x)\wt \P_t(B_0\in dx)\\
&\leq \limsup_{n\to\infty}
 \int \wt \E_t(f \mid B_0=x)\frac{1+\eta}{1-\eta}\wh \P(B_0\in dx)\\
&\leq 
 \limsup_{n\to\infty}
 \int(1+\eps) \wh \E(f \mid B_0=x)\frac{1+\eta}{1-\eta}\wh \P(B_0\in dx)
= (1+\eps) \frac{1+\eta}{1-\eta}\wh \E(f ).
\end{align*}
Similarly, one can use \eqref{rev31.2} and \eqref{rev31.12} to obtain
\begin{align*}
&\liminf_{n\to\infty}\wt \E_t(f )
\geq  (1-\eps) \frac{1-\eta}{1+\eta}\wh \E(f ).
\end{align*}
Since $\eta>0$ can be arbitrarily small, the lemma follows.
\end{proof}

\begin{lemma}\label{rev20.1}
Recall that we have assumed that $\wt \E_{t_1}(f)>0$.
For every $t\geq t_1$, $\wt \E_t(f)>0$.
\end{lemma}

\begin{proof}

It follows from Lemma \ref{a11.4} and \eqref{rev20.2} that the distributions of $B_{t_1}$ under $\wt \P_t$ and $\wt \P_{t_1} $ have strictly positive densities in $\Lambda$, say $p_*(x)$ and $p_{**}(x)$. Let $\Lambda_*\subset\Lambda$ and $c_1,c_2>0$ be such that  $p_*(x)/p_{**}(x)>c_1$ and 
$\wt \E_{t_1}(f \bone_{\{B_{t_1}\in \Lambda_*\}}  ) > c_2$.

According to Remark \ref{rev30.10}, under both $\wt \P_t$ and $\wt \P_{t_1} $,
the process $\{B_s,0\leq s \leq t_1\}$ conditioned by $\{B_0=x, B_{t_1} = y\}$ has the same distribution. Hence
\begin{align*}
\wt \E_t(f) \geq \wt \E_{t}(f \bone_{\{B_{t_1}\in \Lambda_*\}}  )
\geq c_1 \wt \E_{t_1}(f \bone_{\{B_{t_1}\in \Lambda_*\}}  )
\geq c_1 c_2 >0.
\end{align*}
\end{proof}

\subsection{Villemonais' estimate}

Recall that we have assumed that $u$ is fixed and $t_2\leq u+1$. This easily implies that  for some $c_1>0$ and all probability measures $\mu_n$ supported in $\Lambda_1$,
\begin{align}\label{a11.2}
 \P_{\mu_n} (\tau_\Lambda > t_2) \geq c_1,
\end{align}
where $\P_{\mu_n}$ represents the distribution of the driving process $B$ with the initial distribution $\mu_n$. 

Recall definitions of $A_1,A_2$ and $\calH^n_t$ stated in \eqref{m29.10}-\eqref{m31.1} and \eqref{m29.4}.

\begin{lemma}\label{a3.1}
$\lim_{n\to\infty} \P(A_1\cap A_2) = 1$. 
\end{lemma}

\begin{proof}
By Theorem \ref{a3.2} and \eqref{a11.2},
\begin{align}\label{a3.4}
&\P\left(\left|\cH_{t_2}^n\left(\bone_{G_Q}\right) -
	\wt\P_{t_2}(G_Q )\right|
\geq n^{-\alpha d +\gamma(-\alpha+3\delta/4) -\delta}\right)\\
&\qquad\leq 
\E \left(\left|\cH_{t_2}^n\left(\bone_{G_Q}\right) -
	\wt\P_{t_2}(G_Q )\right|\right)
n^{\alpha d -\gamma(-\alpha+3\delta/4) +\delta} \notag \\
&\qquad\leq 2\left(1 + \sqrt{2}\right) c_1^{-1}
n^{-1/2} n^{\alpha d -\gamma(-\alpha+3\delta/4) +\delta}.\notag
\end{align}

Recall from \eqref{a3.3} that we have assumed that $\alpha \leq (1/2 -2\delta+3\gamma \delta/4 )/(\gamma+2d)$. Hence,
 $-1/2+\alpha d-\gamma(-\alpha+3\delta/4) +\delta\leq -\alpha d -\delta$.
This and \eqref{a3.4} show that for some $c_2$,
\begin{align}\label{a3.5}
&\P\left(\left|\cH_{t_2}^n\left(\bone_{G_Q}\right) -
	\wt\P_{t_2}(G_Q )\right|
\geq n^{-\alpha d +\gamma(-\alpha+3\delta/4) -\delta}\right)
\leq c_2 n^{-\alpha d -\delta} .
\end{align}
Since $\Lambda$ is bounded set, the number of $Q$ in $\bQ$ is bounded by $c_3 n^{\alpha d}$. 
It follows from the definition \eqref{m29.10} of $A_1$
and \eqref{a3.5} that
\begin{align}
\P(A_1^c) &\leq
\sum_{Q\in \bQ}
\P\left(\left|\cH_{t_2}^n\left(\bone_{G_Q}\right) -
	\wt\P_{t_2}(G_Q )\right|
\geq n^{-\alpha d +\gamma(-\alpha+3\delta/4) -\delta}\right)
\leq c_3 n^{\alpha d}c_2 n^{-\alpha d -\delta}\notag\\
&= c_3 c_2 n^{-\delta}.\label{a9.1}
\end{align}
Hence, $\lim_{n\to\infty} \P(A_1) = 1$.
The proof for $A_2$  is completely analogous.
\end{proof}

\subsection{Paths close to the boundary and branching}
\label{m31.6}

Our next lemma is comprised of two claims, corresponding to events $C^1_j$ and $C^2_j$ defined below. The first step of the proof is the only common aspect of the two claims.

We believe that some ``unusual'' events are very unlikely to occur at an arbitrarily chosen fixed time $u$. This we cannot prove. But we will prove that for a fixed $u$, there is a  time with this property in $[u,u+1]$. 

Consider any $u>0$, let $\Delta t = n^{-2\alpha + \delta}$, $k_1 = \lfloor 1/\Delta t\rfloor$ and $s_j= u+ j \Delta t$ for $ j=0,\dots, k_1$.
Recall that $J^n$ is the spine of $\X^n$.
Let
\begin{align}\label{rev30.4}
C^1_j &= \{\dist (J^n_{s_j} ,\Lambda^c) \leq 4n^{-\alpha +3\delta /4}\},\\
C^2_j &= \{J^n \text{ passes through a  branch point in  } [s_j, s_{j+1}]\}, \label{rev30.5}\\
C_j &= C^1_j \cup C^2_j .\notag
\end{align}

\begin{lemma}\label{a3.6}
For every $\eps>0$ and $u>0$ there exists $n_1$ so large that for $n\geq n_1$
there exists $j \in\{0,\dots, k_1\}$ such that $\P(C_j)\leq \eps$.
\end{lemma}

\begin{proof}
\emph{Step 1}.
Fix $\eps\in(0,1)$ and $u>0$.
Suppose that for  a given $n$, there is no $j \in\{0,\dots, k_1\}$ such that $\P(C_j)\leq \eps$. Then
\begin{align*}
\E\left( \sum_{1\leq j \leq k_1 } \bone_{C_j}\right) \geq \eps k_1.
\end{align*}
Let 
\begin{align*}
p=\P\left( \sum_{1\leq j \leq k_1} \bone_{C_j} \leq \eps k_1/2\right) .
\end{align*}
Then
\begin{align*}
\E\left( \sum_{1\leq j \leq k_1} \bone_{C_j}\right) \leq p\eps k_1/2 + (1-p)k_1.
\end{align*}
Hence $\eps k_1 \leq p\eps k_1/2 + (1-p)k_1$ and, therefore,
\begin{align*}
&p \leq \frac {1-\eps}{1-\eps/2} <1.
\end{align*}
It follows that
\begin{align*}
\P\left( \sum_{1\leq j \leq k_1} \bone_{C_j} > \eps k_1/2\right) = 1-p
\geq  \frac {\eps/2}{1-\eps/2}>0.
\end{align*}
This implies that at least one of the following inequalities holds,
\begin{align}\label{a3.7}
&\P\left( \sum_{1\leq j \leq k_1} \bone_{C_j^1} > \eps k_1/4\right) 
\geq  \frac {\eps/4}{1-\eps/2}>0,\\
\label{a3.8}
&\P\left( \sum_{1\leq j \leq k_1} \bone_{C_j^2} > \eps k_1/4\right) 
\geq  \frac {\eps/4}{1-\eps/2}>0.
\end{align}
It will suffice to show that each of these inequalities  fails for large $n$.

\medskip
\emph{Step 2}.
Recall that $\calB(v,r)$ denotes a ball with center $v$ and radius $r$.
Let $B$ be  Brownian motion and
\begin{align*}
\wt C^1_j &= \{\dist (B_{s_j} ,\Lambda^c) \leq 4n^{-\alpha +3\delta /4}\}.
\end{align*}

Since $\Lambda$ is a bounded Lipschitz domain, it is easy to see that for some $c_1, c_2>0$ and every $x\in\Lambda$ there exists $z\in \Lambda^c$ such that $|z-x| \leq c_1\dist(x,\Lambda^c)$ and   $\calB( z, c_2 \dist(x,\Lambda^c))\subset \Lambda^c$. 
From here up to and including \eqref{a7.2}, $\P$ will denote the distribution of Brownian motion starting from $x$.
By Brownian scaling, there exists $p_1>0$ such that for all $x$,
\begin{align}\label{a3.9}
&\P\left(\inf\{t>0: B_t \in \Lambda^c\} \leq
\dist(x,\Lambda^c)^2\right)\\
&\qquad\geq \P\left(\inf\{t>0: B_t \in \calB( z, c_2 \dist(x,\Lambda^c))\} \leq
\dist(x,\Lambda^c)^2\right) = p_1.\notag
\end{align}

Let 
\begin{align*}
j_0&= \inf\left\{m\geq 0: \dist( B_{s_m},\Lambda^c) \leq  
4n^{-\alpha +3\delta /4}\right\},\\
j_{i+1}&= \inf\left\{m>j_i: \dist( B_{s_m},\Lambda^c) \leq  
4n^{-\alpha +3\delta /4}, s_m\geq s_{j_i}
+ 4 n^{-2\alpha +3\delta/2}\right\}, \quad i\geq 0.
\end{align*}
By \eqref{a3.9} and the strong Markov property applied at $s_{j_i}$, for every $i$, 
\begin{align*}
&\P\left(\inf\{t>s_{j_i}: B_t \in \Lambda^c\} \leq
s_{j_{i+1}}\right) \geq
 p_1.
\end{align*}
We apply the strong Markov property again to see that for $k\geq 0$,
\begin{align}\label{a3.10}
&\P\left(
\inf\{t>s_0: B_t \in \Lambda^c\} > s_{j_{k+1}}
\right)\\
&\qquad\leq\P\left(
\bigcap_{i=0}^k
\left\{ 
\inf\{t>s_{j_i}: B_t \in \Lambda^c\} > s_{j_{i+1}}
\right\}
\right) \leq
(1- p_1)^{k+1}.\notag
\end{align}

If the event 
$\left\{ \sum_{1\leq j \leq k_1} \bone_{\wt C_j^1} > \eps k_1/4\right\}$
occurred then 
$s_{j_\ell} \leq u+1$, where 
\begin{align*}
\ell = \frac{(\eps k_1/4) \Delta t}{ 4 n^{-2\alpha +3\delta/2}}
= \frac{(\eps  \lceil n^{2\alpha - \delta} \rceil/4) 
n^{-2\alpha + \delta}
}{ 4 n^{-2\alpha +3\delta/2}}
\geq \frac 1 {16} \eps n^{2\alpha -3\delta/2}.
\end{align*}
This and \eqref{a3.10} imply that
\begin{align}\label{a7.2}
&\P\left( \sum_{1\leq j \leq k_1} \bone_{\wt C_j^1} > \eps k_1/6,\ 
\inf\{t>s_0: B_t \in \Lambda^c\} > u+1
\right)\\
&\qquad \leq \P\left(s_{j_\ell} \leq u+1,\ 
\inf\{t>s_0: B_t \in \Lambda^c\} > u+1
\right)\notag\\
&\qquad \leq \P\left(
\inf\{t>s_0: B_t \in \Lambda^c\} > s_{j_\ell}
\right)
\leq (1- p_1)^\ell 
 \leq (1-p_1)^{\eps n^{2\alpha -3\delta/2}/16}.\notag
\end{align}

We will apply methods and results from the proof of Theorem 1.3 in
\cite[p. 688]{BHM00} but   we will use  different notation. 

Let $\Lambda_2$ be such that $\dist(\Lambda_1, \Lambda_2^c) >0$ and $\dist(\Lambda_2, \Lambda^c) >0$.
Let $\H=\{k:  n/4 \leq k\leq n\}$ and $\H^c=\{1,\dots,n\}\setminus \H$. 

The set $\H^c$ is needed in the argument so that we can assume that many processes $X^k$ with $k\in\H^c$ do not jump on the interval $[0,u+2]$; this is represented formally in \eqref{a7.4} below. The processes that do not jump behave like Brownian motion conditioned to stay inside $\Lambda$ and, therefore, they stay ``far'' from the boundary most of the time. This creates an opportunity for  processes $X^k$ with $k\in\H$ to jump ``far'' from the boundary and hence gives them a chance to stay inside $\Lambda$ for the rest of the time interval  $[0,u+2]$. The processes in $\H^c$ are not forgotten---they will be accounted for when we define $\wh \H$ below.

Consider  $c_3 >0$ 
and let $F=F(c_3)$ be the event that at least $c_3 n$ processes $X^k$ with $k\in \H^c$ stay within $\Lambda_2$ on the interval  
$[0,u+2]$. 
Recall that  $X^k_0 \in \Lambda_1$ for all $k$, a.s. 
By the law of large numbers, there exists $c_3=c_3(u) >0$ such that 
\begin{align}\label{a7.4}
\lim_{n\to\infty}\P(F)=1.
\end{align}

Informally, let $M_k$ be the number of branching points on the tree of descendants of $X^k$ on the interval  
$[0,u+2]$. 
Formally, let $M_k$ be the number of branching points on DHPs $H^\ell_{u+2}$ such that $H^\ell_{u+2}(0) = X^k_0$. Every branching point appears on two different DHPs but we count it only once. It follows from \cite[(2.11)]{BHM00} that for every $r>0$ there exists $c^*_r<\infty$ such that for every $k\in\H$ and sufficiently large $n$,
\begin{align}\label{a7.6}
\E (M_k^r\mid F) \leq c^*_r.
\end{align}
The proof of \eqref{a7.6} is based on the branching structure and induction.
Note that \cite[(2.11)]{BHM00} is an upper bound for the $r$-th power of 
``the total number of jumps on the tree of descendants
of particle $m$'' defined just below \cite[(2.10)]{BHM00}. But the definition given in \cite{BHM00} makes it clear that ``jumps'' include branching points on the tree of descendants of $X^k$ defined earlier in this paragraph. Formula \cite[(2.11)]{BHM00} does not contain conditioning on $F$. The conditioning on $F$ is implicit on pages 688--691, as indicated on page 688 of \cite{BHM00}.

Recall $\xi$ introduced in \eqref{a7.1}.
Let $r$ be so large that $\xi r >1$. We have
\begin{align*}
&\P(M_k\geq  n^\xi\mid F) = \P(M_k^r \geq n^{\xi r}\mid F) 
\leq \E(M_k^r \mid F) n^{-\xi r}
\leq c^*_r n^{-\xi r}.
\end{align*}
If we let 
\begin{align*}
K = \bigcup_{ k \in\H} \{M_k\geq  n^\xi\}
\end{align*}
then
\begin{align}\label{a7.5}
&\P\left( K \mid F\right)
\leq c^*_r n^{1-\xi r}.
\end{align}

We construct a
Brownian motion $B^*$  killed on the boundary of $\Lambda$ as follows. Choose $k$ uniformly from $\H$. Follow $X^k$ from time 0 until the first time when $X^k$ exits $\Lambda$ or until the first branching point, whichever comes first. If the process exits $\Lambda$, stop. At a branching point, start following one of the two branches, with equal probabilities, with the random choice independent of $\X^n$. Then apply the inductive construction---follow the branch until it exits $\Lambda$ or the next branching point, whichever comes first. If the process exits $\Lambda$, stop. At a branching point, start following one of the branches, with equal probabilities, with the random choice independent of $\X^n$.

Let $D^*$ be the event on the first line of \eqref{a7.2} but with random objects defined relative to $B^*$ in place of $B$ so that 
\begin{align}\label{a7.3}
\P(D^*) 
 \leq (1-p_1)^{\eps n^{2\alpha -3\delta/2}/16}.
\end{align}

Let $D$ be the event that there is a DHP $H^k_{u+2}$ such that $H^k_{u+2}(0) = X^j_0$ for some $j\in\H$ and
the event on the first line of \eqref{a7.2} holds for this DHP in place of $B$.
If $K^c\cap D$ holds then, in the above notation,  $j$ will be chosen with probability $1/n$, and there will be at most $n^\xi$ choices in the construction of $B^*$ on the interval $[0,u+2]$ so $B^*$ will follow the trajectory of $H^k_{u+2}$ all the way up to $u+2$ with probability greater than or equal to $2^{- n^\xi}$. Thus
\begin{align*}
\P(D^*) \geq \P(D^*\mid K^c \cap D ) \P(K^c \cap D )\geq (1/n) 2^{- n^\xi} \P(K^c \cap D ).
\end{align*}
Hence, by \eqref{a7.3},
\begin{align*}
\P(K^c \cap D ) &
\leq n 2^{ n^\xi}\P(D^*)
\leq n 2^{ n^\xi}
(1-p_1)^{\eps n^{2\alpha -3\delta/2}/16}\\
&= \exp\left(
\log n + n^\xi\log 2 + n^{2\alpha -3\delta/2} \log (1-p_1)\eps/16\right).
\end{align*}
This and our assumption  \eqref{a7.1} that $0<\xi< 2\alpha -3\delta/2$ imply that
\begin{align}\label{a7.7}
\lim_{n\to \infty}
\P(K^c \cap D ) &\leq 
\lim_{n\to \infty}
\exp\left(
\log n + n^\xi\log 2 + n^{2\alpha -3\delta/2} \log (1-p_1)\eps/16\right)=0.
\end{align}

Let  $\wh\H=\{k:  1 \leq k\leq 3n/4\}$ and define $\wh F, \wh K$ and $\wh D$ relative to $\wh \H$ in the same way as $F,K$ and $D$ were defined relative to $\H$. Then, by symmetry and \eqref{a7.4}, \eqref{a7.5} and \eqref{a7.7},
\begin{align}\label{n12.1}
\lim_{n\to\infty}\P(\wh F)=1,\quad
\P\left( \wh K \mid \wh F\right)
\leq c^*_r n^{1-\xi r},\quad
\lim_{n\to \infty}
\P(\wh K^c \cap \wh D )=0.
\end{align}

Note that
\begin{align*}
\left\{ \sum_{1\leq j \leq k_1} \bone_{C_j^1} > \eps k_1/4\right\}
\subset D \cup \wh D,
\end{align*}
so
\begin{align*}
\P&\left( \sum_{1\leq j \leq k_1} \bone_{C_j^1} > \eps k_1/4\right) \\
&\leq \P(K^c \cap D ) + \P(K\cap F)+\P(F^c)
+\P(\wh K^c \cap \wh D ) + \P(\wh K\cap \wh F)+\P(\wh F^c)\\
&\leq \P(K^c \cap D ) + \P(K\mid F)+\P(F^c)
+\P(\wh K^c \cap \wh D ) + \P(\wh K\mid \wh F)+\P(\wh F^c).
\end{align*}
Hence, by \eqref{a7.4}, \eqref{a7.5}, \eqref{a7.7} and \eqref{n12.1},
\begin{align*}
\lim _{n\to \infty} \P\left( \sum_{1\leq j \leq k_1} \bone_{C_j^1} > \eps k_1/4\right) =0,
\end{align*}
contradicting \eqref{a3.7}. This completes the proof that \eqref{a3.7} fails for all large $n$.

\medskip
\emph{Step 3}.
By \eqref{a7.6}, for $\delta >0$ and all $k\in\H$,
\begin{align*}
\P(M_k \geq n^{2\alpha-2\delta}\mid F) \leq \E (M_k^r\mid F) /n^{r(2\alpha-2\delta)}\leq c_r^* n^{-r(2\alpha-2\delta)},
\end{align*}
so
\begin{align*}
\P\left(\bigcup_{k\in\H}\{M_k \geq n^{2\alpha-2\delta}\}\mid F\right) \leq  c_r^* n^{1-r(2\alpha-2\delta)}.
\end{align*}
By \eqref{a7.8}, $2\alpha-2\delta>0$ so we can find $r$ so large that $1-r(2\alpha-2\delta) <0$. It follows that
\begin{align*}
\lim_{n\to\infty}
\P\left(\bigcup_{k\in\H}\{M_k \geq n^{2\alpha-2\delta}\}\mid F\right) =0.
\end{align*}
By \eqref{a7.4},
\begin{align*}
\lim_{n\to\infty}
\P\left(\bigcup_{k\in\H}\{M_k \geq n^{2\alpha-2\delta}\}\right) =0.
\end{align*}
One can prove in the same way that
\begin{align*}
\lim_{n\to\infty}
\P\left(\bigcup_{k\in\wh\H}\{M_k \geq n^{2\alpha-2\delta}\}\right) =0,
\end{align*}
so
\begin{align}\label{a7.9}
\lim_{n\to\infty}
\P\left(\bigcup_{1\leq k \leq n}\{M_k \geq n^{2\alpha-2\delta}\}\right) =0.
\end{align}

The piece of  the spine $\{J^n_t, 0\leq t \leq u+2\} $ is a trajectory within the tree of descendants of some $X^k$ on the interval $[0,u+2]$. Therefore, the number $M_J$ of branching points on $\{J^n_t, 0\leq t \leq u+2\} $ must be less than or equal to $M_k$ for some $k$.
If the event in \eqref{a3.8} holds then $M_J \geq  \eps n^{2\alpha-\delta}/12$. Hence,  \eqref{a7.9} implies that \eqref{a3.8} cannot hold for large $n$.
\end{proof}

\section{A generalization}\label{sec: general}
We will argue that our main result holds for some diffusions in Euclidean domains. However, in our examples, the generator of the diffusion must be associated in an appropriate way with the domain; in other words, we cannot prove that one can choose the generator and the domain independently.

Below $\|\cdot\|$ denotes the Euclidean norm.

\begin{proposition}\label{rev:prop}
 Suppose that $d\ge 1$ and $ U \subset \R^d$ is a domain (open bounded set) with a $C^1$-boundary. 

  Suppose that $\mQ_i:\R^d\to\R^d$, $ \mQ _i=(Q_{i1},Q_{i2},...,Q_{id})^T$, is the gradient of a  $C^2(\R^d,\R)$-function $x_i$ for $1\le i\le d$, and assume that
 $a_{j}:=\sum_{i=1}^d Q_{ij}^2=\|\mQ _i\|^2>0$ on $ U $ for all $1\le i\le d$, and furthermore, that the map $(x_1,...,x_d):\R^d\to\R^d$ is invertible and the matrix $\mathbf{Q}:=(Q_{ij})$ is non-singular at each point of $\R^d$.

Define $H:\R^d\to (0,\infty)$ by $$H(u_1,...,u_d)=\prod_{i=1}^d \|\mQ _i(u_1,...,u_d)\|,$$
and the operator $$\mathcal{L}=\frac{1}{2H}\nabla\cdot H\mathbf{C}\nabla$$
on $U$, where 
$\mathbf{C}:\R^d\to\R^{d\times d}$, and $\mathbf{C}$ is diagonal
with $c_{ii}:=1/a_{i}=\|\mQ _i\|^{-2}, 1\le i\le d$.

Let  $\lambda$ denote the principal eigenvalue of $\mathcal{L}$  on $ U $ with Dirichlet boundary conditions and let $\phi_{\mathcal{L},\lambda}>0$ denote the principal eigenfunction associated with 
$\mathcal{L}$ and $\lambda$ on $ U $.
Let  $\mathcal{L}^{\phi_{\mathcal{L},\lambda}}$
denote the generator of the $\calL$-diffusion conditioned to stay in $U$ forever, obtained by ``Doob's $h$-conditioning'' in space-time by the parabolic function $(u,t)\to \phi_\calL(u) e^{\lambda t}$.

Consider a sequence of Fleming-Viot processes
driven by $\calL$-diffusions and assume that their empirical initial distributions converge weakly to a measure supported on a compact subset of $U$. The spines of these processes converge weakly to a diffusion corresponding to $\mathcal{L}^{\phi_{\mathcal{L},\lambda}}$.
\end{proposition}

\begin{proof} (Sketch) As in the statement of the proposition, the  points in $U$ will be denoted by $(u_1,u_2,...,u_d)$. By assumption, $\mQ _i:\R^d\to\R^d$ is the gradient of some $C^2(\R^d)$-function $x_i=x_i(u_1,...,u_d)$,
 that is,
$$\mQ _i=(Q_{i1},Q_{i2},...,Q_{id})=\left(\frac{\partial x_i}{\partial u_{1}},...,\frac{\partial x_i}{\partial u_{d}} \right).$$
Define the map $\mathbf{x}:\R^d\to\R^d$
by $\mathbf{x}=(x_1,...,x_d)$;
by assumption, it has an inverse map  $\mathbf{u}:\R^d\to\R^d$.
Informally speaking, we have two systems  of coordinates on $\R^d$, which we will call the $x$-system and the $u$-system.

Let $D=  \x(U) $ and
note that $D$ is a $C^1$-domain because, by assumption, $\mathbf{Q}$ is non-singular on $\R^d$ so, using the Inverse Function Theorem, the map $\mathbf{x}:\R^d\to\R^d$
is a local diffeomorphism, under which the image of a $C^1$-domain  is a $C^1$-domain again.

Let $L=\Delta/2$ be the operator in the $x$-system  and let $\phi$ denote the principal eigenfunction associated with $L$ with Dirichlet boundary conditions on $ D $.
Let  $L^\phi$
denote the generator of Brownian motion conditioned to stay in $D$ forever.

As is well known (or a straightforward calculation reveals), the operator $\mathcal{L}=\frac{1}{2H}\nabla\cdot H\mathbf{C}\nabla$ in the $u$-system is the pullback of  $L$ in the $x$-system, where $H$ and $\mathbf{C}$ are defined via the non-singular matrix-function $\mathbf{Q}$ as in the statement of the proposition. This means  that $(Lf)(x)=(\mathcal{L}g)(u)$ for a twice differentiable $f$ where $g$ is the pullback of $f$ (i.e., $g(u):=f(x(u))$).

We will argue that the following claims hold true.
\begin{enumerate}
\item[(1)] The principal eigenvalues of the two operators agree.
\item[(2)] The pullback of $\phi$ is $\phi_{\mathcal{L},\lambda}$.
\item[(3)] 
The pullback of $L^{\phi}$ is $\mathcal{L}^{\phi_{\mathcal{L},\lambda}}$.
\item[(4)] Brownian motion (with killing at $\partial D$) in the $x$-system becomes an $\mathcal{L}$-diffusion (with killing at $\partial U $) in the  $u$-system; the $L^{\phi}$-diffusion in the $x$-system becomes
an $\mathcal{L}^{\phi_{\mathcal{L},\lambda}}$-diffusion in the $u$-system. 
\end{enumerate}
These are all invariance properties (under coordinate transforms): (1) follows from the fact that
\begin{align*}
\lambda(L;D)&=\inf \{l\in\R\mid \exists\ 0<f\in C^2(D): (L-l)f=0\},\\
\lambda(\mathcal{L}; U )&=\inf \{l\in\R\mid \exists\ 0<g\in C^2( U ): (\mathcal{L}-l)g=0\},
\end{align*}
while
(2) follows because the infimum is in fact the minimum and for the minimal value the function $f$ (function $g$) is unique.

Next, (3) follows because $L^{\phi}(\,\cdot\,)=\frac{1}{\phi}L(\phi(\, \cdot\,))-\lambda$ and an analogous formula holds for $\mathcal{L}^{\phi_{\mathcal{L},\lambda}}$, whereas
(4) follows from (1)-(3) along with the fact that diffusions are uniquely determined via  the corresponding martingale problems.

We conclude that the Fleming-Viot system on $D$ driven by Brownian motion describes the same process as the Fleming-Viot system on $ U $ driven by the $\mathcal{L}$-diffusion. 
In fact, we have a {\it coupling} of the two particle systems. In particular, the ``$n$-spine'' must refer to the same process: the unique infinite lineage. 
Hence, Theorem \ref{a1.7} implies that  the claim concerning the limiting distribution of $n$-spines is applicable to  the Fleming-Viot system on $ U $ driven by the $\mathcal{L}$-diffusion. We leave it to the reader to check that the notion of weak convergence is  invariant under the smooth coordinate transformation, and a similar remark applies to the statements involving the initial empirical measures.
\end{proof}

\subsection{Examples} Below we give some examples of domains and operators which can be generated using Proposition \ref{rev:prop}.

1.  In $d=1$, our setup includes an arbitrary interval $ U =(\alpha,\beta)$ with an operator
    $$\mathcal{L}=\frac{1}{2H}\frac{d}{du}\frac{1}{H}\frac{d}{du},$$
    where $H=|Q|$ and $Q$ satisfies the assumptions of the theorem.
For example, we can take  $x(u)=\text{sgn}(u)|u|^a$, $ a>0$. This yields $Q(u)=ax^{a-1}$ on $ U $ whenever $\alpha>1$. Thus, writing $\gamma:=1-a<1$, the operator becomes the time-changed Bessel-generator $$(\mathcal{L}f)(u)=\frac{u^{2\gamma}}{2(1-\gamma)^{2}}\left(f''(u)+\frac{\gamma}{u} f'(u)\right).$$
    
2. Let $d=2$. Then the proposition applies to all $C^1$-domains in $\R^2$ and all operators
    $$\mathcal{L}=\frac{1}{2H}\nabla\cdot H\mathbf{C}\nabla$$
on $U$, where 
$$\mathbf{C}=\begin{bmatrix} \|\nabla F\|& 0\\
0&\|\nabla G\|\end{bmatrix}^{-2},$$  
    $$H(u_1,u_2)= \|\nabla F\|^2\|\nabla G\|^2,$$
    and $(F,G)\in C^2(\R^2,\R^2)$ is invertible and such that
    $$\frac{\partial F}{\partial u_{1}}\frac{\partial G}{\partial u_{2}}\neq 
       \frac{\partial F}{\partial u_{2}}\frac{\partial G}{\partial u_{1}}.$$
       
3. Continuing the previous example, a more concrete class of operators for $d=2$ can be obtained  by choosing
       a $C^1$-domain $ U \subset (0,\infty)^2$ and $\alpha,\beta,\gamma,\delta\in\R$
       such that $\alpha\delta\neq \beta\gamma$ and 
       $$F(u,v)=\exp(\alpha u+\beta v),\quad G(u,v)=\exp(\gamma u+\delta v),$$
       yielding the operator $\mathcal{L}$ on $ U $ for which
       \begin{align*}
          \left(\mathcal{L}f\right)(u,v)=
          \frac{1}{2}e^{-2[(\alpha+\gamma) u+(\beta+\delta)v]}\nabla\cdot e^{2[(\alpha+\gamma) u+(\beta+\delta)v]}\,\mathbf{C}\nabla f(u,v),
       \end{align*}
       where $$\mathbf{C}:=\begin{bmatrix} (\alpha^2+\beta^2)e^{-2\alpha u-2\beta v} & 0\\0 & (\gamma^2+\delta^2)e^{-2\gamma u-2\delta v}\end{bmatrix}.$$
       This simplifies to
        \begin{align*}
         \left(\mathcal{L}f\right)(u,v)&=
          (\alpha^2+\beta^2)e^{-2(\alpha u+\beta v)}\left(\partial^2 f/\partial u^2+\gamma\,\partial f/\partial u\right)\\
&\qquad+
         (\gamma^2+\delta^2)e^{-2(\gamma u+\delta v)}\left(\partial^2 f/\partial v^2+\beta\,\partial f/\partial v\right).
          \end{align*}
          For example, when $(\alpha,\beta)$ and $(\gamma,\delta)$ are unit vectors in $\R^2$ and $w_1=\alpha u+\beta v$  and $w_2=\gamma u+\delta v$, then
         \begin{align*}
         \mathcal{L}=e^{-2w_{1}}\left(\partial^2 /\partial u^2+\gamma\,\partial /\partial u\right)+e^{-2w_{2}}\left(\partial^2 /\partial v^2+\beta\,\partial /\partial v\right).
          \end{align*}

\section{Spine for the superprocess version of our Fleming-Viot process}\label{sec:super}
We will briefly sketch an argument which shows that, in a sense, our main theorem has an analog for superprocesses. See \cite{AE} for an introduction to superprocesses.
Let us point out an unfortunate terminological inconsistency---a  Fleming-Viot superprocess (see, e.g., \cite{AE})
is not a natural superprocess analogue of our Fleming-Viot process
 because in the Fleming-Viot superprocess model the death rate does not depend on the spatial position of the particle, whereas in our model deaths occur only at $\partial D$. 
Intuitively, when working with superprocesses, one has already passed to the infinite limit with the  population size so individual particles have infinitesimally small masses.

Consider $\eps>0$ (the branching rate parameter) and let $X$ be a superprocess on a bounded Euclidean domain $D$ corresponding to the semilinear elliptic operator $Lu-\eps u^2$ on $D$, where $L$ is a second order elliptic operator with smooth coefficients and the underlying motion is the $L$-diffusion  with killing at $\partial D$. 
This process becomes extinct almost surely, but the so-called $Q$-process, denoted $X^{Q,\eps}$,  is well defined for all time. To construct  $X^{Q,\eps}$ one conditions $X$ to survive until time $t>0$ and then lets $t\to\infty$. Next we ``normalize'' $X^{Q,\eps}$ as follows, 
$$X^{\eps}_t:=\frac{X^{Q,\eps}_{t}}{\|X^{Q,\eps}_{t}\|}.$$

A  problem analogous to the one considered in this paper is to decide if we can identify a unique immortal particle for $X^\eps$ and then describe its distribution as $\eps\to 0$.
To justify the analogy, assume that  the intensity of the branching parameter $\eps$  is small.  If one  completely ignores the effect of  branching inside the domain,  the particle picture  is simple:  the (infinitesimal) particles are killed at the boundary $\partial D$, and at the same time their mass is instantaneously re-distributed without spatial bias in $D$, i.e., the birth/branching locations are uniformly chosen. Of course, the above picture is only approximate because $\eps\neq0$.

A result analogous to our  Theorem \ref{a1.7} would be that $X^{\eps}$  has a unique immortal particle with law denoted by $\QQ^{\eps}$ and $\lim_{\eps\to 0}\QQ^{\eps}=\QQ$, where $\QQ$ is the law of a diffusion corresponding to the elliptic operator $L^{\phi}$. 
This statement is indeed true.  In fact,  $X^{Q,\eps}$ has a unique immortal particle with the \emph{same} law  $\QQ$ for {\it every} $\eps>0$ by the spine decomposition that can be found, for example, in the recent preprint \cite{LRSS22} and references therein. It is shown that the $Q$-process can be  decomposed into an immortal ``infinitesimal'' particle (the ``spine'') and a collection of  ``local bushes'' which become extinct \cite[Lemmas 4.11 and
4.17]{LRSS22}. (See also \cite{RenSongSun} for  proofs.) The spine corresponds to the $h$-transformed operator $L^{\phi}$, where $\phi$ is the positive eigenfunction (ground state) for the principal eigenvalue of $L$ on $D$ with zero Dirichlet boundary conditions.
Finally, we note that 
the result holds for state spaces and operators much more general than described above; see again \cite{LRSS22}  and  references therein.

\section{Acknowledgments}

We are grateful to Rodrigo Ba\~nuelos for very useful advice.
We thank the referees for very detailed reports and many suggestions for improvement.

\bibliographystyle{imsart-number} 



\end{document}